\documentclass[11pt,a4paper,leqno]{article}
\usepackage[DIV=11]{typearea}
\usepackage[english]{babel}
\usepackage[utf8]{inputenc}
\usepackage{hyperref}
\usepackage{algpseudocode}
\usepackage{amsmath}
\usepackage{amssymb}
\usepackage{amsfonts}
\usepackage{amsthm}
\usepackage{amscd}
\usepackage{mathdots}
\usepackage{tikz}
\usepackage{pgfplots}
\usetikzlibrary{positioning}
\pgfplotsset{compat=1.8}
\usepackage{subfig}
\usepackage{enumerate}
\usepackage{bbold}
\usepackage{fancyhdr}
\usepackage[noblocks]{authblk}
\usepackage{verbatim}
\makeatletter
\def\revddots{\mathinner{\mkern1mu\raise\p@
\vbox{\kern7\p@\hbox{.}}\mkern2mu
\raise4\p@\hbox{.}\mkern1mu\raise7\p@\hbox{.}\mkern1mu}}
\makeatother

\theoremstyle{plain}
\newtheorem{thm}{Theorem}[section]
\newtheorem{lem}[thm]{Lemma}
\newtheorem{prop}[thm]{Proposition}
\newtheorem{cor}[thm]{Corollary}

\theoremstyle{definition}

\newtheorem{exmp}[thm]{Example}
\newtheorem{rem}[thm]{Remark}

\theoremstyle{remark}

\newcommand{\F}{F}
\newcommand{\GL}{\textnormal{GL}}

\newcommand{\C}{\mathbb{C}}

\newcommand{\Ind}{\mathrm{Ind}}
\newcommand{\Hom}{\mathrm{Hom}}

\newcommand{\Irr}{{\mathrm{Irr}}}

\newcommand{\drep}[2]{\delta(#1,#2)}
\newcommand{\zrep}[2]{\zeta(#1,#2)}

\title{Theta correspondence for symplectic--orthogonal and metaplectic--orthogonal p-adic dual pairs}
\author{Petar Baki\'c, Marcela Hanzer}
\date{}
\begin{document}

\maketitle

\begin{abstract} In this paper, we completely describe the Howe correspondence for the dual pairs from the title over a nonarchimedean local field of characteristic zero. More specifically, for every irreducible admissible representation of these groups,  we find its first occurrence index in the  theta correspondence and we describe, in terms of their Langlands parameters, the small theta lifts on all levels.
\end{abstract}
\section{Introduction}
In this paper, we study the Howe correspondence for the dual pairs  $(\textnormal{Sp}(n), \textnormal{O}(V))$, and $(\textnormal{Mp}(n),\textnormal{O}(V))$ over a nonarchimedean field $F$ of characteristic zero. Here $\textnormal{Sp}(n)$ denotes the symplectic and $\textnormal{Mp}(n)$ the metaplectic group of rank $n/2$, for any even integer $n\geq 2$; $\textnormal{O}(V)$ denotes the orthogonal group of the quadratic space which belongs to a fixed (but arbitrary) Witt tower. Let us briefly recall the basic setting.

Let $V$ be a quadratic space of dimension $m$ over $F$ (i.e.~a space endowed with a symmetric non-degenerate $F$-bilinear form) and let $\textnormal{O}(V)$ denote the corresponding orthogonal group. The groups $\textnormal{O}(V)$ and $\textnormal{Sp}(n)$ form a dual pair inside $\textnormal{Sp}(nm)$ (cf.~\cite{Kudla1}, Chapter II).  Now, for $F$ p-adic, there exists a unique two-fold central covering group of $\textnormal{Sp}(n),$ the metaplectic group $\textnormal{Mp}(n).$  By fixing a non-trivial additive character $\psi$ of $F$, we obtain the so-called Weil representation $\omega_{nm,\psi}$ of the metaplectic group $\textnormal{Mp}(nm).$ The covering in  $\textnormal{Mp}(nm)$ splits over the dual pair $(\textnormal{O}(V),Sp(n))$ if $m=\dim V$ is even, thus, by restricting this representation to $\textnormal{O}(V) \times \textnormal{Sp}(n)$ we obtain the Weil representation $\omega_{m,n,\psi}$ of this dual pair; if $m$ is odd, we get the Weil representation of $\textnormal{O}(V) \times \textnormal{Mp}(n).$ We now let $(G,H)$ denote one of these pairs (note that we allow $G$ to be either metaplectic/symplectic or orthogonal).

Now for any irreducible admissible representation $\pi$ of $G$ we may look at the maximal $\pi$-isotypic quotient of $\omega_{m,n,\psi}$. We denote it by $\Theta(\pi,m)$ and call it the full theta lift of $\pi$ to $V$. This representation, when non-zero, has a unique irreducible quotient, denoted $\theta(\pi,m)$---the small theta lift of $\pi$. This basic fact, called the Howe duality conjecture, was first formulated by Howe \cite{Howe_theta_series}, proven by Waldspurger \cite{Waldspurger_howe_duality} (for odd residue characteristic) and by Gan and Takeda \cite{Gan_Takeda_proof_of_Howe} in general.

The Howe duality establishes a map $\pi \mapsto \theta(\pi)$ which is called the theta correspondence. It turns out that it is very important in  local representation theory, but also, its global counterpart (which agrees with local theta correspondence) gives us one of a very few direct ways to explicitly construct automorphic forms. The study of theta correspondence was started by Roger Howe \cite{Howe_theta_series,Howe_transcending}, and further developed by Kudla \cite{Kudla2,Kudla1}, Rallis \cite{rallis1984howe}, Kudla-Rallis \cite{KR1}, Moeglin-Vigneras-Waldspurger \cite{MVW_Howe}, Waldspurger \cite{Waldspurger_howe_duality} and many others. In recent years, the study of theta correspondence is heavily influenced by the rapid developments in the framework of the Langlands program, so that new results are formulated in this spirit (cf.~\cite{Gan_Savin_Metaplectic_2012}, \cite{Atobe_Gan}).

We now describe our results. Let us fix an irreducible admissible representation $\pi$ of a group $G$ as above. Then, each $H$ (the other member of the dual pair $(G,H)$) belongs to a certain Witt tower. There is a natural “other'' Witt tower attached to this one (cf.~\cite{Kudla1}, Chapter V or Section \ref{subs_first_occurrence} here). Because of the conservation relation (cf.~\cite{Sun_Zhu_conservation}), it is natural to simultaneously study the lifts of $\pi$ on both of these towers. We find the first occurrence index of $\pi$ in each of these towers (Theorem \ref{theorem_occurrence}), and describe the Langlands parameters of all the non-zero lifts of $\pi$ in these towers (so, specifically, we describe the lift of $\pi$ as a representation of $H$)---this is obtained in Theorem \ref{theorem_appearance}.

Among the two target towers we consider, the tower in which $\pi$ occurs on the lower level is referred to as the ``going-down'' tower  for $\pi$ (the other one is the ``going-up tower'', cf.~Section 4.2). Let us write $\pi$ as the Langlands quotient, $\pi=L(\nu^{s_r}\delta_r,\ldots,\nu^{s_1}\delta_1;\tau)$---see Section 2.4 for notation. It is interesting to note that, if we fix a pair of target towers, then the going-down tower for $\pi$ (in that pair of towers) is also the going-down tower for the tempered representation $\tau$ (Section 5). An interesting phenomenon is that $\pi$ can occur “earlier'' than $\tau$ in the going-down tower (we define $l(\pi)$---the relative first occurrence index---in Section 5).

Measuring that discrepancy between occurrences of $\pi$ and of $\tau$ is precisely what gives us the first occurrence index for $\pi$
(we note that we heavily rely on the results of Atobe and Gan \cite{Atobe_Gan} which describe the lifts of the tempered representations). 
The main idea is finding the longest ladder-representation (cf.~\cite{Lapid_Minguez_determinantal_Tadic}) of a specific kind (the twists of the Steinberg representation) inside $\nu^{s_r}\delta_r\times\nu^{s_{r-1}}\delta_{r-1}\times \cdots \times  \nu^{s_1}\delta_1$. Each rung of this ladder leads to $\pi$ occurring ``lower'' in the going-down tower (Lemma \ref{lemma_reverse_temp_thin}, Theorem \ref{theorem_occurrence}). This also gives us the description of the lifts themselves. Namely, we determine both the occurrence index and the description of the lifts using an algorithm which gives us the ladder-representation mentioned above. We prove both the main property of the algorithm (cf.~Proposition \ref{prop_unique_q}) and the description of the “deeper'' lifts (Theorem \ref{theorem_appearance} (1)) by induction, roughly, on the number of the rungs. It is worth noting that once the first occurrence of the representation $\pi$ is known, one can derive the description of the lifts directly, without relying on the induction. However, that kind of approach is lengthier and does not use the features of the algorithm which are already proved. 

We now describe the content of this paper. In the Preliminaries section we review the classical (and metaplectic) groups we are interested in, the Witt towers, and  the form of their parabolic subgroups. We recall the Langlands classification for the irreducible admissible representations of these groups, as well as their Jacquet modules. We also briefly recall the local Langlands correspondence, just as much as we need to use the results of \cite{Atobe_Gan}. The third section is the key technical part of the paper. In it, we introduce the algorithm which we use throughout the paper. This algorithm has two roles: on one hand, it finds the longest ladder as mentioned above; on the other hand, starting from the standard representation attached to $\pi$, it finds another representation which also possess $\pi$ as the unique irreducible quotient, but which is more appropriate for use in the theta correspondence than the standard representation. We also give several examples to further explain the algorithm and to comment on its reversibility and its resemblance to M\oe glin-Waldspurger algorithm (\cite{MW_Zelevinski}). Then, in the fourth section, we review the relevant facts concerning theta correspondence. In the fifth section, using the preparation done in the third section, we find the first occurrence index of a given representation $\pi$ in the given Witt tower. Finally, in the sixth section, we give the explicit description of the theta lifts of a representation $\pi,$ for both Witt towers in the pair. This is again proved using ladders and induction. We also provide an example to illustrate how the induction process works. In this example, we assume that the representation $\tau$ mentioned above is square-integrable---this simplifies the arguments which are obscured by technical difficulties when $\tau$ is tempered but not square-integrable. In the appendix, we prove two auxiliary lemmas which were stated in the the sixth section, but whose proofs were postponed in order to streamline the exposition of the main arguments.

\medskip
\noindent{\textbf{Acknowledgements}}\\
We are grateful to Wee Teck Gan for his warm hospitality at  the conference On the Langlands Program: Endoscopy and Beyond, NUS, IMS, Singapore, where we discussed some of the topics of this paper.
This work is  supported in part by Croatian Science Foundation under the project IP-2018-01-3628.

\section{Preliminaries}

\subsection{Groups}
\label{subs:groups}
Let $F$ be a nonarchimedean local field of characteristic $0$ and let $|\cdot|$ be the absolute value on $F$ (normalized as usual). The groups considered in this paper will be defined over $F$. For $\epsilon = \pm 1$ fixed, we let
\[
\begin{cases}
W_n = \text{a }(-\epsilon)\text{-Hermitian space of dimension }n\; \text{over } F,\\
V_m = \text{an }\epsilon\text{-Hermitian space of dimension }m\; \text{over } F.
\end{cases}
\]
When $\epsilon = 1$, this means that $W_n$ is symplectic, whereas $V_m$ is a quadratic space (we do not consider unitary groups in this paper). We consider isometry groups attached to the pair $(W_n,V_m)$.  We thus set
\[
G_n = G(W_n) = 
\begin{cases}
\text{the metaplectic group }\textnormal{Mp}(W_n), \quad &\text{if } \epsilon = 1 \text{ and } m \text{ is odd},\\
\text{the isometry group of } W_n \quad &\text{otherwise}
\end{cases}
\]
and define $H(V_m)$ similarly by switching the roles of $W_n$ and $V_m$. Here $\textnormal{Mp}(W_n)$ denotes the unique non-trivial two-fold covering of $Sp(W_n)$; cf.~\cite{Kudla1}, \cite{MVW_Howe}. Thus, from now on we use $G_n=G(W_n)$ (or $H(V_m)$) to denote $\textnormal{Sp}(n), \textnormal{Mp}(n), \textnormal{O}(m)$; this way, the notation is unified.
Furthermore, if $X$ is a vector space over $F$, we denote by $\textnormal{GL} (X)$ the general linear group of $X$. Note that all the groups defined here are totally disconnected locally compact topological groups.

\subsection{Witt towers}
\label{subs:Witt}
Every $\epsilon$-Hermitian space $V_m$ has a Witt decomposition
\[
V_m = V_{m_0} + V_{r,r} \quad (m = m_0 + 2r),
\]
where $V_{m_0}$ is anisotropic and $V_{r,r}$ is split (i.e.~a sum of $r$ hyperbolic planes). The space $V_{m_0}$ is unique up to isomorphism, and so is the number $r \geqslant 0$, which is called the Witt index of $V_m$. The collection of spaces
\[
\mathcal{V} = \{V_{m_0} + V_{r,r} \colon r \geqslant 0\}
\]
is called a Witt tower. Since, for the quadratic spaces, we have
\[
\det(V_{m_0+2r}) = (-1)^r\det(V_{m_0}) \in F^\times/(F^\times)^2,
\]
the quadratic character
\[
\chi_V(x) = (x, (-1)^{\frac{m(m-1)}{2}}\det(V))_F
\]
is the same for all the spaces $V$ in a single Witt tower (see \cite[\S V.1]{Kudla1}); here $(\cdot, \cdot)_F$ denotes the Hilbert symbol. In the case when $V$ is symplectic, we take $\chi_V$ to be the trivial character.

\subsection{Parabolic subgroups}
\label{subs:parabolic}
First, let $V_m$ be a quadratic space of dimension $m$. We may choose a subset $\{v_1, \dotsc, v_r,\allowbreak v_1',\dotsc, v_r'\}$ of $V_{m}$ such that $(v_i,v_j) = (v_i',v_j')=0$ and $(v_i,v_j') = \delta_{ij}$. Here $r$ denotes the Witt index of $V_m.$ We let $B=TU$ denote the standard $F$-rational Borel subgroup of $H(V_m)$, i.e.~the subgroup of $H(V_m)$ stabilizing the flag
\[
0 \subset \text{span}\{v_1\} \subset \text{span}\{v_1,v_2\}\subset \dotsb \subset \text{span}\{v_1,v_2\dotsc,v_r\}.
\]
Furthermore, for any $t \leq r$ we set $U_t = \text{span}\{v_1,\dotsc, v_t\}$ and $U_t' = \text{span}\{v_1',\dotsc, v_t'\}$; we can then decompose
\[
V_m = U_t \oplus V_{m-2t} \oplus U_t'
\]
The subgroup $Q_t$ of $H(V_m)$ which stabilizes $U_t$ is a maximal parabolic subgroup of $H(V_m)$; it has a Levi decomposition $Q_t = M_tN_t$, where $M_t = \textnormal{GL}(U_t) \times H(V_{m-2t})$ is the Levi component, i.e.~the subgroup of $Q_t$ which stabilizes $U_t'$. We often identify $\textnormal{GL}(U_t)$ with $\textnormal{GL}_t(F)$.

By letting $t$ vary, we obtain a set $\{Q_t: t \in\{1,\dotsc,r\}\}$ of standard maximal parabolic subgroups. By further partitioning $t$, we get the rest of the standard parabolic subgroups---generally, the Levi factor of a standard parabolic subgroup is of the form
\[
\textnormal{GL}_{t_1}(F) \times \dotsm \times \textnormal{GL}_{t_k}(F) \times H(V_{m-2t})  \quad (t = t_1 +\dotsb + t_k).
\]
The parabolic subgroups of $\textnormal{Sp}(W_n)$ are constructed in a similar fashion. Finally, the notion of parabolic subgroups is naturally extended to the case when $H(V_m)$ is the metaplectic group $\text{Mp}(V_m)$; see e.g.~III.2 in \cite{Kudla1}, \cite{Szpruch_PhD}. We denote the maximal standard parabolic subgroups of $G(W_n)$ and $H(V_m)$ by $P_t$ and $Q_t$, respectively.

\subsection{Representations}
\label{subs:reps}
Let $G = G(W_n)$ be one of the groups described in \S\ref{subs:groups}. By a representation of $G$ we mean a pair $(\pi,V)$ where $V$ is a complex vector space and $\pi$ is a homomorphism $G \to \textnormal{GL}(V)$. With $V_{\infty}$ we denote the subspace of $V$ comprised of all the smooth vectors, i.e.~those having an open stabilizer in $G$. Note that $V_{\infty}$ is a subrepresentation of $V.$
If $V = V_{\infty}$, we say that the representation $(\pi,V)$ is smooth. Unless otherwise stated, we will assume that all the representations are smooth; the category of all smooth complex representations of $G$ will be denoted by $\mathcal{A}(G)$. The set of equivalence classes of irreducible representations of $G$ will be denoted by $\Irr(G)$.

For each parabolic subgroup $P=MN$ of $G$ we have the (normalized) induction and localization (Jacquet) functors, $\Ind_P^G\colon \mathcal{A}(M) \to \mathcal{A}(G)$ and $R_P \colon \mathcal{A}(G) \to \mathcal{A}(M)$. These are related by the standard Frobenius reciprocity
\[
\Hom_G(\pi, \Ind_P^G(\pi')) \cong \Hom_M(R_P(\pi), \pi')
\]
and by the second (Bernstein) form of Frobenius reciprocity,
\[
\Hom_G(\Ind_P^G(\pi'), \pi) \cong \Hom_M(\pi', R_{\overline{P}}(\pi))
\]
(here $\overline{P} = M\overline{N}$ is the parabolic subgroup opposite to $P$).
If $P=MN$ is a parabolic subgroup of $G(W_n)$ with Levi factor $M = \textnormal{GL}_{t_1}(F) \times \dotsm \times \textnormal{GL}_{t_k}(F) \times G(W_{n-2t})$, we use
\[
\tau_1 \times \dotsm \times \tau_k \rtimes \pi_0
\]
to denote $\Ind_P^G(\tau_1 \otimes \dotsm \otimes \tau_k \otimes \pi_0)$, where $\tau_i$ is a representation of $\textnormal{GL}_{t_i}(F)$, $i = 1,\dotsc,k$, and $\pi_0$ is a representation of $G(W_{n-2t})$ (with $t = t_1 +\dotsb+t_k$). We use  analogous (Zelevinsky)  notation for the parabolic induction for the  general linear groups.

This is a slight abuse of notation in the case of metaplectic groups, since a Levi subgroup is not necessarily the product of general linear factors and a smaller metaplectic group, cf.~\cite{Szpruch_PhD}, the second section, or \cite{Hanzer_Muic_metaplectic_Jacquet}, the second section. In this case, we view $\tau_i$ as a representation of the two-fold cover of  $\textnormal{GL}_{t_i}(F),$ denoted by $\widetilde{\textnormal{GL}_{t_i}(F)},$ obtained by twisting a representation of $\textnormal{GL}_{t_i}(F)$ by the (genuine) character $\chi_{\psi}(g,\epsilon) = \epsilon \gamma(det g,\frac{1}{2}\psi)^{-1}$ of $\widetilde{\textnormal{GL}_{t_i}(F)}.$ Here $\psi$ is a non-trivial additive character of $F$ which will appear in the Howe correspondence, and $\gamma$ is related to the Weil index of a character of second degree, cf. \cite{Kudla1}, I.4. The same applies for the Jacquet functor for the metaplectic groups, cf.~\cite{Hanzer_Muic_metaplectic_Jacquet}, 4.2.

To obtain a complete list of irreducible representations of $G(W_n)$, we use the Langlands classification. Let $\delta_i \in \textnormal{GL}_{t_i}(F), i = 1,\dotsc, r$ be irreducible discrete series representations, and let $\tau$ be an irreducible tempered representation of $G(W_{n-2t})$, where $t=t_1+\dotsb+t_r$. Any representation of the form
\[
\nu^{s_r}\delta_r \times \dotsb \times \nu^{s_1}\delta_1 \rtimes \tau,
\]
where $s_r \geqslant \dotsb \geqslant s_1 > 0$ (and where $\nu$ denotes the character $\lvert\det\rvert$ of the corresponding general linear group) is called a standard representation (or a standard module). It possesses a unique irreducible quotient, the so-called Langlands quotient, denoted by $L(\nu^{s_r}\delta_r, \dotsc, \allowbreak \nu^{s_1}\delta_1; \allowbreak \tau)$. Conversely, every irreducible representation can be represented as the Langlands quotient of a unique standard representation. In this way, we obtain a complete description of $\Irr(G(W_n))$. The Langlands classification is also valid for the metaplectic groups in an analogous form, cf.~\cite{Ban_Jantzen-Langl_class}.

We will use this (quotient) form of the Langlands classification interchangeably with the subrepresentation form, by means of the Gefand-Kazhdan results for general linear groups and the M\oe glin-Vigneras-Waldspurger involution through the following lemma (see \cite[Lemma 2.2]{Atobe_Gan}).

\begin{lem}
\label{lemma:MVWinv}
Let $\tau_i \in \Irr(\textnormal{GL}_{t_i}(F))$, $i=1,\dotsc,r$ and $\pi_0 \in \Irr(G(W_{n_0}))$. Let $P$ be a standard parabolic subgroup of $G(W_n)$ ($n = n_0 + 2\sum t_i$) with Levi component equal to $\textnormal{GL}_{t_1}(F) \times \dotsm \times \textnormal{GL}_{t_r}(F) \rtimes G(W_{n_0})$. Then, for any $\pi \in \Irr(G(W_{n}))$, the following statements are equivalent:
\begin{enumerate}[(i)]
\item $\pi \hookrightarrow \tau_1 \times \dotsm \times \tau_r \rtimes \pi_0$; 
\item $\tau_1^\vee \times \dotsm \times \tau_r^\vee \rtimes \pi_0 \twoheadrightarrow \pi$.
\end{enumerate}
Similarly, let $P'$ be a standard parabolic subgroup of $\textnormal{GL}_{m}(F)$ ($m = \sum t_i$) with Levi component equal to $\textnormal{GL}_{t_1}(F) \times \dotsm \times \textnormal{GL}_{t_r}(F)$. Then, for any $\pi' \in \Irr(\textnormal{GL}_{m}(F))$ the following statements are equivalent:
\begin{enumerate}
\item[(iii)]	$\pi' \hookrightarrow \tau_1 \times \dotsm \times \tau_r$; 
\item[(iv)] $\tau_r \times \dotsm \times \tau_1 \twoheadrightarrow \pi'.$
\end{enumerate}
\end{lem}
Here, as well as in the rest of the paper, $\tau^\vee$ denotes the contragredient representation.

\subsection{Local Langlands Correspondence}
\label{subs:LLC}
Another way of classifying the irreducible representations of $G(W_n)$ is by means of the Local Langlands Correspondence (LLC). We use it mainly to harvest the results on lifts of tempered representations established by Atobe and Gan in \cite{Atobe_Gan}. Without going into detail, we give a brief description of the basic features of LLC; a concise overview of the theory along with the key references can be found in appendices A and B of \cite{Atobe_Gan}.

The LLC parametrizes $\Irr(G(W_n))$ by representations of the Weil-Deligne group, $\mathrm{WD}_F = \mathrm{W}_F \times \textnormal{SL}_2(\mathbb{C})$ (here $W_F$ denotes the Weil group of $F$). More precisely, we define $\Phi(G(W_n))$, for any $n$, as a set of (equivalence classes of) admissible homomorphisms:
\[
\begin{cases}
\Phi(\textnormal{O}(W_n)) = \{\phi \colon \mathrm{WD}_F \to \textnormal{Sp}(n-1,\mathbb{C})\} / \cong, \quad \text{if } n \text{ is odd},\\
\Phi(\textnormal{Sp}(W_n)) = \{\phi \colon \mathrm{WD}_F \to \textnormal{SO}(n+1,\mathbb{C})\} / \cong,\\
\Phi(\textnormal{O}(W_n)) = \{\phi \colon \mathrm{WD}_F \to \textnormal{O}(n,\mathbb{C}) | \det(\phi) = \chi_W\} / \cong \quad \text{if } n \text{ is even},\\
\Phi(\textnormal{Mp}(W_n)) = \{\phi \colon \mathrm{WD}_F \to \textnormal{Sp}(n,\mathbb{C})\} / \cong.
\end{cases}
\]
The irreducible representations of $G(W_n)$ are then parametrized by the so-called $L$-pa\-ra\-me\-ters, i.e.~pairs of the form $(\phi,\eta)$, where $\phi \in \Phi(G(W_n))$, and $\eta$ is a character of the (finite) component group of the centralizer of $\text{Im}(\phi)$. The set of representations which correspond to the same $\phi$ is called the $L$-packet attached to $\phi$. 

Any $\phi \in \Phi(G(W_n))$ can be decomposed as 
\[
\phi = \bigoplus_{n \geqslant 1} \phi_n \otimes S_n, 
\]
where $\phi_n$ is a representation of $W_F$, whereas $S_n$ denotes the unique algebraic representation of $\textnormal{SL}_2(\C)$ of dimension $n$. For an irreducible representation $\rho \otimes S_n$ of $ \mathrm{WD}_F,$ we denote by $m_{\phi}(\rho \otimes S_n)$  its multiplicity in the parameter $\phi.$ Tempered representations are parametrized by pairs $(\phi,\eta)$ in which $\phi(W_F)$ is bounded (for full odd orthogonal groups we need an extra ingredient $\nu\in\{1,-1\})$; cf.~\cite{Atobe_Gan}, Section 3); among those, the multiplicity free parameters of correct parity correspond to discrete series representations.

Note that, unlike $\phi$, the choice of $\eta$ is non-canonical: it depends on the choice of a Whittaker datum of $G(W_n)$ which we fix in the manner explained in  \cite[Remark B.2]{Atobe_Gan}. For metaplectic groups, it also depends on the choice of an additive character $\psi$ related to theta correspondence; cf.~\cite{Gan_Savin_Metaplectic_2012}, \cite{Atobe_Gan}, B.4 and Theorem B.8.

\subsection{Computing Jacquet modules}
\label{subs:computingJM}
On a number of occasions we shall need to compute the Jacquet modules of various representations. We let $R_n, n\geq 0$ denote the Grothendieck group of admissible representations of $\GL_n(\F)$ of finite length; we also set $
R = \oplus_{n\geq 0} R_n$.

For $\pi_1 \in \Irr(\GL_{n_1}(\F))$ and $\pi_2 \in \Irr(\GL_{n_2}(\F))$ the pairing
\[
(\pi_1, \pi_2) \mapsto \pi_1 \times \pi_2
\]
defines an additive mapping $\times \colon R_{n_1} \times R_{n_2} \to R_{n_1+n_2}$. We extend the mapping $\times$ to a multiplication on $R$ in a natural way.

On the other hand, for any $\pi \in \Irr(\GL_n(\F))$ we may identify $R_{P_k}(\pi)$ with its semi-simplification in $R_k \otimes R_{n-k} \hookrightarrow R\otimes R$ (here $P_k$ temporarily denotes the $k$-th maximal standard parabolic subgroup of $\GL_n(\F)$). We define
\[
m^*(\pi) = (\mathbb{1}\otimes \pi) \oplus (\pi \otimes \mathbb{1})  \oplus  \sum_{k=1}^{n-1}R_{P_{k}}(\pi) \in R\otimes R
\]
and extend $m^*$ to an additive map $R \to R\otimes R$. So, from now on, ``$=$" denotes the equality in the appropriate Grothendieck group. The basic fact due to Zelevinsky (see Section 1.7 of \cite{Zelevinsky_GL} for additional details) is that
\begin{equation}
\label{eq_mstar}
m^*(\pi_1 \times \pi_2) = m^*(\pi_1) \times m^*(\pi_2).
\end{equation}
In most cases, we will consider $m^*(\pi)$ when $\pi = \delta([\nu^a\rho, \nu^b\rho])$, i.e.~the discrete series representation attached to the segment $[\nu^a\rho, \nu^b\rho]$, or $\pi = \zeta(\nu^a\rho, \nu^b\rho)$, i.e.~the Langlands quotient of $\nu^b\rho \times \dotsm \times \nu^a\rho$, where $b-a \in \mathbb{Z}_{\geq 0}$ (we review these representations in \S\ref{subs_aux}). In those cases, we have
\begin{equation}
\label{eq_mstar2}
\begin{aligned}
m^*(\delta([\nu^a\rho, \nu^b\rho])) &= \sum_{i=a-1}^b \delta([\nu^{i+1}\rho, \nu^b\rho]) \otimes \delta([\nu^a\rho, \nu^i\rho]),\\
m^*(\zeta(\nu^a\rho, \nu^b\rho)) &= \sum_{i=a-1}^b \zeta(\nu^a\rho, \nu^i\rho) \otimes \zeta(\nu^{i+1}\rho, \nu^b\rho).
\end{aligned}
\end{equation}
In the above equations we set $\delta([\nu^c\rho,\nu^{c-1}]) = \zeta(\nu^{c},\nu^{c-1}) = \mathbb{1} \in \Irr(\GL_0(\F))$ for any $c\in \mathbb{R}$.

This theory was extended by Tadi\'{c} to the case of classical groups in \cite{Tad_struc}. For any $\pi \in \Irr(G_n)$ we let $\mu^*(\pi)$ be the sum of the semi-simplifications of $R_P(\pi)$ when $P$ varies over the set of standard parabolic subgroups of $G_n$. The relevant formula is now
\begin{equation}
\label{eq_tadic_classical}
\mu^*(\delta \rtimes \pi) = M^*(\delta) \rtimes \mu^*(\pi).
\end{equation}
The definition of $M^*$ can be found in \cite[Theorem 5.4]{Tad_struc}, but we shall need it here only in the special case when $\delta = \delta([\nu^a\rho, \nu^b\rho])$ or $\zeta(\nu^a\rho, \nu^b\rho)$; in these cases, we have (\cite[\S 14]{tadic2012reducibility})
\begin{equation}
\label{eq_tadic_classical2}
\begin{aligned}
M^*(\delta([\nu^a\rho, \nu^b\rho])) &= \sum_{i=a-1}^b\sum_{j=i}^b \delta([\nu^{-i}\rho^\vee,\nu^{-a}\rho^\vee]) \times \delta([\nu^{j+1}\rho,\nu^b\rho]) \otimes \delta([\nu^{i+1}\rho, \nu^j\rho])\\
M^*(\zeta(\nu^a\rho, \nu^b\rho)) &=  \sum_{i=a-1}^b\sum_{j=i}^b \zrep{\nu^{-b}\rho^\vee}{\nu^{-(j+1)}\rho^\vee} \times \zrep{\nu^{a}\rho}{\nu^{i}\rho} \otimes \zrep{\nu^{i+1}\rho}{\nu^j\rho}.
\end{aligned}
\end{equation}

\section{Rearranging the standard module}
One of the key steps in our approach is a careful rearrangement of the standard module. In order to justify it, we first prove some auxiliary results.

\subsection{Auxiliary results on irreducibility}
\label{subs_aux}
Recall that irreducible essentially discrete series representations of $\textnormal{GL}(F)$ correspond to segments of cuspidal representations $[\nu^a\rho,\nu^b\rho]$ with $b-a \in \mathbb{Z}_{\geq 0}$ and $\rho$ a unitary cuspidal representation. More precisely, for any such segment, the representation
\begin{equation}
\label{eq_segment}
\nu^b\rho \times \nu^{b-1}\rho \times \dotsm \times \nu^a\rho
\end{equation}
possesses a unique irreducible subrepresentation denoted $\delta([\nu^a\rho,\nu^b\rho])$. This is an essentially discrete series representation. Conversely, any irreducible essentially discrete series representation corresponds to a unique segment in this way. It is important to note that \eqref{eq_segment} also has a unique irreducible quotient (i.e.~the Langlands quotient), which we denote $\zeta(\nu^a\rho,\nu^b\rho)$. The key facts concerning $\delta$ and $\zeta$ may be found in Zelevinsky's paper \cite{Zelevinsky_GL}. Throughout this section, we freely use Zelevinsky's terminology and results on linked segments.

We begin by examining the relation between $\delta$ and $\zeta$. We say that the segments $[\nu^a\rho,\nu^b\rho]$ and $[\nu^{a'}\rho',\nu^{b'}\rho']$ are adjacent if $\rho = \rho'$, and $a' = b+1$ or $b' = a+1$.
\begin{lem}
\label{lemma_zel0}
Let $[\nu^a\rho,\nu^b\rho]$ and $[\nu^{a'}\rho',\nu^{b'}\rho']$ be segments which are not adjacent. Then
\[
\delta([\nu^a\rho,\nu^b\rho]) \times \zeta(\nu^{a'}\rho',\nu^{b'}\rho') \quad \text{and} \quad \zeta(\nu^{a'}\rho',\nu^{b'}\rho') \times \delta([\nu^a\rho,\nu^b\rho])
\]
are irreducible and isomorphic.
\end{lem}
\begin{proof}
This is a special case of Lemma I.6.3 of \cite{Moeglin1989}.
\end{proof}

We are only interested in segments defined by $\rho = \mathbb{1}$, the trivial representation of $\GL_1(F)$. We therefore adjust our notation and set
\[
\drep{a}{b} = \delta([|\cdot|^a,|\cdot|^b]) \quad \text{and} \quad \zrep{a}{b} = \zeta([|\cdot|^a,|\cdot|^b]) 
\]
for any $a\leq b$ such that $b-a \in \mathbb{Z}$.
We now examine the case in which the segments are adjacent.
\begin{lem}
\label{lem_zel1}
Let $[a,b]$ and $[b+1,d]$ be adjacent segments. Then the representation
\[
\zrep{b+1}{d} \times \drep{a}{b}
\]
is of length two. Furthermore,
\begin{itemize}
\item its unique irreducible quotient is the Langlands quotient of
\[
|\cdot|^d \times |\cdot|^{d-1} \times \dotsm \times  |\cdot|^{b+1} \times \drep{a}{b};
\]
it is also the unique irreducible quotient of $\drep{a}{b-1}\times\zrep{b}{d}$.
\item Its unique irreducible subrepresentation is the Langlands quotient of
\[
|\cdot|^d \times |\cdot|^{d-1} \times \dotsm \times  |\cdot|^{b+2} \times \drep{a}{b+1};
\]
at the same time, it is the unique irreducible quotient of $\drep{a}{b}\times\zrep{b+1}{d}$.
\end{itemize} 
\end{lem}
\begin{proof}
In this case, the representation $\pi = \zrep{b+1}{d} \times \drep{a}{b}$ reduces. Moreover, formulas \eqref{eq_mstar} and \eqref{eq_mstar2} show that $R_{P_1}(\pi)$ is of lentgh two, so $\pi$ is also of length two. We leave the proof of the second part of the lemma to the reader.
\end{proof}
We also make use of the following lemma. If $[b,d]$ and $[d+1,e]$ are adjacent segments, the above lemma shows that $\zrep{d+1}{e} \times \drep{b}{d}$ has a unique irreducible quotient, which we now denote by $L$.
\begin{lem}
\label{lem_zel2}
\begin{enumerate}[(i)]
\item For any $a \leq b$, the representation $L \times \drep{a}{d}$ is irreducible. In particular, $L \times \drep{a}{d} {\cong} \drep{a}{d} \times L$.
\item For any $c \geq b$, the representation $L \times \drep{c}{d-1}$ is irreducible. In particular, $L \times \drep{c}{d-1} {\cong} \drep{c}{d-1} \times L$.
\end{enumerate}
\end{lem}
\begin{proof}
We first prove (ii). Denote by $\xi$ the unique (Langlands) quotient of
\[
\zrep{d+1}{e} \times \drep{b}{d} \times \drep{c}{d-1}.
\]
Since $L \times \drep{c}{d-1}$ is a quotient of the above, and $\xi$ is the unique irreducible quotient, we also have $L \times \drep{c}{d-1} \twoheadrightarrow \xi$. We now have
\begin{align*}
\zrep{d+1}{e} \times \drep{b}{d} \times \drep{c}{d-1} &\cong \zrep{d+1}{e} \times \drep{c}{d-1}  \times \drep{b}{d} \\
&\cong \drep{c}{d-1}  \times \zrep{d+1}{e} \times \drep{b}{d},
\end{align*}
where the first isomorphism follows from the fact that $[b,d]$ contains $[c,d-1]$ (so they are not linked), and the second from Lemma \ref{lemma_zel0}.
Since $\xi$ is the unique quotient of the above representation, it must also be a quotient of $\drep{c}{d-1}  \times L$. Therefore, $\xi$ is the unique quotient of both $\drep{c}{d-1}  \times L$ and $L \times \drep{c}{d-1}$. Since it appears with multiplicity one ($\xi$ being the Langlands quotient), from Lemma \ref{lemma:MVWinv} it follows that $\drep{c}{d-1}  \times L$ and $L \times \drep{c}{d-1}$ are irreducible and isomorphic.

We now prove (i) in a similar manner. Let $\xi$ be the unique (Langlands) quotient of
\[
\zrep{d+1}{e} \times \drep{b}{d} \times \drep{a}{d}.
\]
As in (ii), we also have $L \times \drep{a}{d} \twoheadrightarrow \xi$. By Lemma \ref{lem_zel1}, this implies
\[
\drep{b}{d-1} \times \zrep{d}{e} \times \drep{a}{d} \twoheadrightarrow \xi.
\]
We now have
\begin{align*}
\drep{b}{d-1} \times \zrep{d}{e} \times \drep{a}{d} &\cong \drep{b}{d-1} \times \drep{a}{d} \times \zrep{d}{e}  \\
&\cong  \drep{a}{d} \times \drep{b}{d-1} \times \zrep{d}{e},
\end{align*}
where the first isomorphism follows from Lemma \ref{lemma_zel0}, and the second from the fact that $[a,d]$ contains $[b,d-1]$ (so they are not linked).
Therefore, $\xi$ is an irreducible quotient of $\drep{a}{d} \times \drep{b}{d-1} \times \zrep{d}{e}$, only this time, we do not know if $\xi$ is unique. Nevertheless, Lemma \ref{lem_zel1} now shows that $\xi$ is necessarily a quotient of $\drep{a}{d} \times L$ or of $\drep{a}{d} \times \zrep{d}{e} \times \drep{b}{d-1}$. If the latter were true, we would have
\[
\zrep{d}{e}\times \drep{a}{d} \times \drep{b}{d-1} \cong \drep{a}{d} \times \zrep{d}{e} \times \drep{b}{d-1} \twoheadrightarrow \xi,
\]
contradicting the shape of the standard module for $\xi$. Therefore, $\xi$ is a quotient of $\drep{a}{d} \times L$. As we already know $L \times \drep{a}{d} \twoheadrightarrow \xi$, and $\xi$ appears with multiplicity one (again, using the multiplicity one property of the Langlands quotient), the conclusion follows.
\end{proof}

\subsection{The algorithm}
\label{subs_algorithm}
Recall that any $\pi \in \Irr(G_n)$ is the unique quotient of a (unique) standard representation
\begin{equation}
\label{osnovna}
\delta_r \times \delta_{r-1} \times \dotsm \times \delta_1 \rtimes \tau.
\end{equation}
Here $\tau$ is an irreducible tempered representation, and $\delta_i = \delta([\rho\nu^{a_i},\rho\nu^{b_i}]), i=1,\dotsc,r$ are irreducible essentially square integrable $\GL$-representations with $a_r+b_r \geq \dotsb \geq a_1+b_1 >0$. We will only be working with representations defined by $\rho = \mathbb{1}$ and $a_i, b_i$ from a fixed class modulo $\mathbb{Z}$. We denote this class $\alpha + \mathbb{Z}$ for some representative $\alpha \in \mathbb{R}$. In fact, we will always have either $\alpha=0$ or $\alpha = \frac{1}{2}$, that is, $a_i$ and $b_i$ will be integers or half-integers.

Representations defined by non-linked segments may freely switch places by Lemma \ref{lemma_zel0}. We may therefore group all the $\GL$-representations defined by $\rho = \mathbb{1}$ and numbers $a_i,b_i \in \alpha + \mathbb{Z}$; we call this the $\alpha$-block of the standard module. Furthermore, it is easy to see that linkedness properties allow us to sort the representations inside the $\alpha$-block with respect to the lexicographic order on the segments:
\begin{equation}
\label{eq_order}
[a,b] < [c,d] \iff b < d \text{ or } (b = d \text{ and } a < c).
\end{equation}
In other words, we push the segments ending in larger numbers to the left; if two segments end in the same number, the shorter one goes further left.
Throughout the rest of this section, we assume that the class $\alpha + \mathbb{Z}$ is fixed, that all segments belonging to this class and $\rho = \mathbb{1}$ are grouped, and that they are sorted this way. We use one more convention in the rest of the paper: if $a = b+1$, we set $\drep{a}{b} = \zrep{a}{b} = \mathbb{1}_{\GL_0}$. This enables uniform notation even when some of the segments we consider become empty. We point out that any segment of the form $[a+1,a]$ has a uniquely determined position with respect to the ordering \eqref{eq_order}, and that this position is different from the one taken by $[a'+1,a']$ for $a' \neq a$.

We are now ready to describe our algorithm. We apply it to the $\alpha$-block of the standard module of $\pi \in \Irr(G_n)$. Before we start, let us note that in line 5, $L(\zrep{e+1}{k} \times \drep{b}{e})$ denotes the unique irreducible quotient of $\zrep{e+1}{k} \times \drep{b}{e}$, i.e.~the Langlands quotient of $|\cdot|^{k} \times \dotsc \times |\cdot|^{e+1} \times \drep{b}{e}$.
\begin{algorithmic}[1]
\State Input: $k \in \alpha + \mathbb{Z}$
\State Variables: L, $b$, $e$
\State Let $[b_0,k]$ be the shortest segment ending in $k$. Set $b=b_0$, $e=k$.
\Loop
\State $\text{L} \gets L(\zrep{e+1}{k} \times \drep{b}{e})$ \Comment{initially $\text{L}=\drep{b_0}{k}$ and $\zrep{e+1}{k}$ is empty!}
\State Replace $\zrep{e+1}{k} \times \drep{b}{e}$ with $\text{L}$
\State L switches places with all segments $[a,e]$, where $a \leq b$ 
\State L switches places with all segments $[c,e-1]$, where $c \geq b$
\State Replace L with $\drep{b}{e-1} \times \zrep{e}{k}$
\State $a \gets \max\{a': a' < b \text{ such that the segment }[a',e-1]\text{ occurs and, if } e \neq \frac{3}{2}, a' \neq 2-e\}$
\Comment{if there are no such segments, \textbf{exit}}
\State $e \gets (e-1)$
\State $b \gets a$
\EndLoop
\end{algorithmic}
The output of this algorithm is a representation of the form
\begin{equation}
\label{eq_output}
\drep{a_r}{b_r} \times \drep{a_{r-1}}{b_{r-1}} \times \dotsm \times \drep{a_j}{b_j} \times \zeta(b_j+1,k) \times \drep{a_{j-1}}{b_{j-1}} \times \dotsm \times \drep{a_1}{b_1} \rtimes \tau.
\end{equation}
Notice that the segments $[a_r,b_r], [a_{r-1},b_{r-1}], \dotsc, [a_j,b_j], [a_{j-1},b_{j-1}], \dotsc [a_1,b_1]$ remain ordered decreasingly with respect to \eqref{eq_order}. In the next subsection, we show that $\pi$ is still the unique irreducible quotient of this representation (obtained by applying the algorithm to the standard module of $\pi$).

\begin{exmp}
\label{ex_alg}
Let us demonstrate the algorithm on a particular example. Let $\pi$ be given by the following standard module:
\[
\drep{3}{6} \times |\cdot|^5 \times \drep{4}{5} \times \drep{3}{4} \times \drep{2}{4} \times  |\cdot|^3 \times \drep{1}{3} \times  |\cdot|^1 \rtimes \tau.
\]
Also set $k=5$; thus  $b_0 = e = 5$ in line 3. We now go through the first iteration the loop:

\noindent 5: $\text{L} = |\cdot|^5$

\noindent 6: nothing happens

\noindent 7: $\drep{3}{6} \times \drep{4}{5} \times |\cdot|^5 \times \drep{3}{4} \times \drep{2}{4} \times  |\cdot|^3 \times \drep{1}{3} \times  |\cdot|^1 \rtimes \tau$

\noindent 8: nothing happens

\noindent 9: nothing happens

\medskip

\noindent Lines 10-12 now set $a = 3$, $e = 4$, $b=3$; we start the second iteration:

\noindent 5: $\text{L} = L(|\cdot|^5 \times \drep{3}{4})$

\noindent 6: $\drep{3}{6} \times \drep{4}{5} \times \text{L} \times \drep{2}{4} \times  |\cdot|^3 \times \drep{1}{3} \times  |\cdot|^1 \rtimes \tau$

\noindent 7: $\drep{3}{6} \times \drep{4}{5} \times \drep{2}{4}  \times \text{L} \times |\cdot|^3\times \drep{1}{3} \times  |\cdot|^1 \rtimes \tau$

\noindent 8: $\drep{3}{6} \times \drep{4}{5} \times \drep{2}{4} \times |\cdot|^3 \times \text{L} \times \drep{1}{3} \times  |\cdot|^1 \rtimes \tau$

\noindent 9: $\drep{3}{6} \times \drep{4}{5} \times \drep{2}{4} \times |\cdot|^3 \times |\cdot|^3 \times \zrep{4}{5} \times \drep{1}{3} \times  |\cdot|^1 \rtimes \tau$

\medskip

\noindent Lines 10-12 set $a = 1$, $e = 3$, $b=1$ and we start the third iteration:

\noindent 5: $\text{L} = L(\zrep{3}{4} \times \drep{1}{3})$

\noindent 6: $\drep{3}{6} \times \drep{4}{5} \times \drep{2}{4} \times |\cdot|^3 \times |\cdot|^3 \times L \times  |\cdot|^1 \rtimes \tau$

\noindent 7: nothing happens

\noindent 8: nothing happens

\noindent 9: $\drep{3}{6} \times \drep{4}{5} \times \drep{2}{4} \times |\cdot|^3 \times |\cdot|^3\times \drep{1}{2} \times \zrep{3}{5} \times  |\cdot|^1 \rtimes \tau$

\medskip

\noindent After this iteration, we encounter the \textbf{exit} command in line 10. Thus, the algorithm is completed, and the resulting representation is
\[
\drep{3}{6} \times \drep{4}{5} \times \drep{2}{4} \times |\cdot|^3 \times |\cdot|^3\times \drep{1}{2} \times \zrep{3}{5} \times  |\cdot|^1 \rtimes \tau.
\]
\end{exmp}
We close this subsection with a few remarks on the algorithm.
\begin{rem}
\label{rem_alg}
\begin{enumerate}[(i)]
\item Note that there are two possible reasons for the algorithm to exit in line 10: either there are no segments ending in $e-1$ remaining, or all such segments are in fact equal to $[2-e,e-1]$.
\item Even after the final iteration, there may be some factors $\delta\nu^s$ remaining between $\zrep{e}{k}$ and $\tau$, just like in the above example, where we have
\[
\dotsm \times \zrep{3}{5} \times |\cdot|^1 \rtimes \tau.
\]
We point out that the algorithm ensures that $\zrep{e}{k}$ may switch places with any of those remaining representations apart from $[2-e,e-1]$: the segments defining those representations end in $e-2$ or lower, so they are not linked to $[e,k]$.
\item Another way to express the result of the above algorithm is by considering the so-called ladders in the $\alpha$-block. Consider a sequence of segments $[c_1,d_1],[c_2,d_2],\dotsc,[c_t,d_t]$ of segments appearing in the $\alpha$-block such that
\[
d_{i-1} = d_{i}-1 \quad \text{and} \quad c_{i-1} \lneqq c_{i}, \quad \forall i = 2,\dotsc, t.
\]
We call such a sequence a ladder (note that the term ``ladder'' usually describes a somewhat larger class of representations, cf.~\cite{Lapid_Minguez_determinantal_Tadic}). Take the longest ladder such that $[c_t,d_t]= [b_0,k]$ (we keep the notation $b_0$ and $k$ from the algorithm: $[b_0,k]$ is the shortest segment ending in $k$) and $c_1 + d_{1} \gneqq 1$ unless $c_1 = d_1 = \frac{1}{2}$. If there are several such ladders of maximal length, choose the one which minimizes the width of the segments, i.e.~which is minimal with respect to the lexicographic ordering of vectors
\[
(d_t-c_t, \dotsc, d_2-c_2,d_1-c_1).
\]
Now the result of the algorithm can be expressed by the following transformation: 
\begin{itemize}
\item  for all $i$, replace $\drep{c_{i}}{d_{i}}$ with $\drep{c_{i}}{d_{i}-1}$ so that the $\alpha$-block remains sorted.
\item Insert $\zrep{k-t+1}{k} = \zrep{d_1}{d_t}$ immediately to the right of all factors of the form $\drep{c_1}{d_1-1}$.
\end{itemize}
\item The algorithm is reversible, and therefore injective. In other words, the standard representation which transforms into \eqref{eq_output} is unique, if it exists. To reconstruct the original standard representation from the representation \eqref{eq_output} obtained by the algorithm, we may also use the ladder description instead of running the algorithm backwards. We choose a  ladder $[c_1',d_1'],[c_2',d_2'],\dotsc,[c_t',d_t']$ with $t = k-b_j$ from the sequence $[a_j,b_j],\dotsc,[a_r,b_r]$ such that $[c_1',d_{1}'] = [a_j,b_j]$. Again, $d_{i-1}'=d_i'-1$; all  the ladders in this paper will have this property. If there is no ladder of length $k-b_j$, take the longest available ladder and use empty segments of the form $[d+1,d],[d+2,d+1]\dotsc$ to achieve length $k-b_j$. Also, if there is more than one ladder of length $k-b_j$, choose the one which \emph{maximizes} segment widths, i.e.~which is maximal with respect to the lexicographic ordering of vectors
\[
(d_1'-c_1', d_2'-c_2',\dotsc, d_t'-c_t').
\]
Then replace $\drep{c_i'}{d_i'}$ with $\drep{c_i'}{d_i'+1}$ (for each $i=1,2,\dotsc,k-b_j$), so that the $\alpha$-block remains sorted, and remove $\zrep{b_j+1}{k}$.

Let us prove that the above procedure is inverse to the one described in (iii). Let $\pi \in \Irr(G_n)$ and let $[c_1,d_1],\dotsc, [c_t,d_t]$ be the ladder described in (iii). To prove that the above procedure reverses the algorithm, we need to show that $[c_i',d_i'] = [c_i,d_i-1]$ for all $i=1,\dotsc,t$. We do this by induction.

By definition, $[c_1',d_1'] = [a_j,b_j]$; on the other hand, line 9 of the algorithm guarantees that $[a_j,b_j]$ in (3) is in fact equal to $[c_1,d_1-1]$. Thus $[c_1',d_1'] = [c_1,d_1-1]$; this settles the base case. Now assume that $[c_i',d_i'] = [c_i,d_i-1]$ for all $i=1,\dotsc,s-1$; we prove that $[c_s',d_s'] = [c_s,d_s-1]$. Indeed, $[c_s',d_s']$ is chosen among $[a_j,b_j],\dotsc,[a_r,b_r]$ as the longest segment in which $d_s' = d_{s-1}'+1 = d_s-1$ and which is linked to $[c_{s-1}',d_{s-1}'] = [c_{s-1},d_{s-1}-1]$. Obviously $[c_{s},d_s-1]$ satisfies these two conditions (the choice of ladder in (iii) implies $[c_{s-1},d_{s-1}]$ and $[c_s,d_s]$ are linked). We need to prove that it is the longest among such segments. If there were a longer segment, say $[c,d]$, with $d = d_s-1 = d_{s-1}$ and such that $[c,d]$ is linked to $[c_{s-1}',d_{s-1}'] = [c_{s-1},d_{s-1}-1]$, the fact that it is longer than $[c_s,d_s-1]$ implies $c < c_s$, and the linkedness implies $c_{s-1} < c$. This also means that $[c,d]$ hasn't been shortened by the algorithm. However, $c_{s-1} < c < c_s$ and $d = d_s-1$ would imply that $[c,d]$ would be chosen instead of $[c_{s-1},d_{s-1}]$ as the $(s-1)$-th rung of the ladder in (iii). Therefore, $[c,d]$ does not exist, i.e.~$[c_s,d_s-1]$ is indeed the longest segment satisfying the required properties. This completes the proof.

\item Not all representations of the form \eqref{eq_output} can be obtained as the output of our algorithm (for some appropriately chosen input). First, if $(a_j,b_j) = (a_{j-1},b_{j-1})$, then \eqref{eq_output} is obviously not an output of our algorithm (this is ensured by line 8). If $(a_j,b_j) \neq (a_{j-1},b_{j-1})$, then we may apply the ladder transformation described in part (iv) of this remark. If $[c_t,k]$ is not the shortest segment ending in $k$ after this procedure, then the above representation obviously cannot be obtained by applying the algorithm to some standard module. In all other cases, \eqref{eq_output} is the output which corresponds to the standard module obtained by this inverse transformation.

\item One can easily see that our algorithm, which deals only with the blocks of representations of general linear groups, is actually the first step of M\oe glin-Waldspurger algorithm for the determination of the Zelevinsky-Aubert dual of an irreducible representation of a general linear group (cf.~\cite{MW_Zelevinski}, II.2 and \cite{Badulescu_Renard_Zelevinsky}, Section 1); there, the role that $\pi$ has in the present  article is played by $L(\delta_r \times \delta_{r-1} \times \dotsm \times \delta_1)$ (cf.~\eqref{osnovna}). We expect that some parts of the results in subsection 3.3 below can  be derived through their results, but, in the end, one has to pass from the general linear groups to classical groups  in the Jacquet module calculation in  Proposition \ref{prop_unique_q}. So, we keep our arguments self-contained and we do not refer to the M\oe glin-Waldspurger algorithm any more.
\end{enumerate}
\end{rem}
\subsection{Explaining the algorithm}
We now explain, line by line, why $\pi$ is still the unique irreducible quotient of the resulting representation. Our proof of this fact proceeds by induction.

Assume that $\pi$ is the unique irreducible quotient of the representation obtained after performing line 6 in some iteration of the loop. This is certainly true in the base case, i.e. in the first iteration, as line 6 does not change anything the first time we go through the loop. Line 7 is then justified by Lemma \ref{lem_zel2} (i), whereas Lemma \ref{lem_zel2} (ii) explains line 8. Lemma \ref{lem_zel1} shows that $L$ is a (the) quotient of $\drep{b}{e-1} \times \zrep{e}{k}$, so $\pi$ is indeed a quotient of the representation obtained after line 9. 

The main technical question is whether the uniqueness is preserved by line $9$. If $\pi$ is still the unique quotient after line 9, we have the following two cases. If the algorithm ends in line 10, then we are done. If not, then after line 10, $\pi$ is the unique quotient of the representation which contains:
\[
\dotsb \times \drep{b}{e-1}\times \zrep{e}{k} \times \drep{a}{e-1} \times \dotsb
\]
Thus, the irreducible subquotient of $\zrep{e}{k} \times \drep{a}{e-1}$ which participates in this epimorphism must be its unique irreducible quotient, i.e.~$L(\zrep{e}{k} \times \drep{a}{e-1})$. This justifies line 6 in the next iteration of the loop (after taking into account lines 11 and 12), thereby completing the induction step. Therefore, it remains to prove the following.

\begin{prop}
\label{prop_unique_q}
The representation obtained after line 9 in any iteration of the loop has a unique irreducible quotient, $\pi$. Moreover, $\pi$ appears with multiplicity one in this representation.
\end{prop}

\begin{proof}
For simplicity of notation, assume that the $\GL$-part of the standard module consists only of the $\alpha$-block (the proof remains the same if there are other factors in the $\GL$-block). Let $\Pi$ denote the representation obtained after line 9 in some iteration of the algorithm. Then $\Pi$ is of the following form:
\begin{equation}
\label{eq_line9}
\drep{a_r}{b_r} \times \drep{a_{r-1}}{b_{r-1}} \times \dotsm \times \drep{a_j}{b_j} \times \zeta(b_j+1,k) \times \drep{a_{j-1}}{b_{j-1}} \times \dotsm \times \drep{a_1}{b_1} \rtimes \tau.
\end{equation}
Here $a_i,b_i \in \alpha + \mathbb{Z}, \forall i$. Furthermore, because of the way we organized the $\alpha$-block, we know that $[a_{i-1},b_{i-1}] \leq [a_{i},b_{i}]$ with respect to the lexicographic order \eqref{eq_order}. Note that $k-b_j$ signifies the number of iterations our algorithm has gone through up to this point. 

The proof proceeds by induction. Rather than inducing on the number of iterations, we induce on $r-j$, i.e.~the number of factors appearing to the left of $\zrep{b_j+1}{k}$. Let us explain the approach. Assume that $\xi$ is an irreducible subquotient of $R_P(\Pi)$ for some suitable standard parabolic $P$, satisfying the following conditions:
\begin{itemize}
\item if $\pi_1$ is any irreducible quotient of $\Pi$, then $R_P(\pi_1)$ contains $\xi$
\item $\xi$ appears with multiplicity one in $R_P(\Pi)$.
\end{itemize}
If such $\xi$ exists, our claim obviously follows. We prove the existence of such $\xi$ by induction. The base case is covered by the fact that any standard representation has a unique irreducible quotient, which comes with multiplicity one. To perform the inductive step, we consider the following cases:
\begin{enumerate}[1.]
\item $k \neq b_r$
\item $k = b_r$
\begin{itemize}
\item[2.1] in the initial iteration, $b_0 \lneqq k$ 
\item[2.2] 
in the initial iteration, $b_0 = k$.
\end{itemize}
\end{enumerate}
Let us find $\xi$ in Case 1. For simplicity, let $a = a_r$ and $b=b_r$. Note that $[a_{r-1},b_{r-1}]$ may also be equal to $[a,b]$. Therefore, let $m$ be the total number of segments equal to $[a,b]$ appearing in $\Pi$, so that $[a_r,b_r] = [a_{r-1},b_{r-1}] = \dotsb = [a_{r-m+1}, b_{r-m+1}] = [a,b]$, and $[a_{r-m},b_{r-m}] \lneqq [a,b]$. We let $P$ be the standard maximal parabolic with $m\cdot (b-a+1)$ the size of the $\GL$-block in its Levi factor. If we denote with $\Pi'$ the part of $\Pi$ which remains after removing the leading $\drep{a}{b}$'s and let $\Delta = \drep{a}{b} \times \dotsm \times \drep{a}{b}$ ($m$ times), we have
\[
\Pi = \Delta  \rtimes \Pi'.
\]
It is easy to see that $\Pi'$ can be obtained by applying the appropriate number of iterations of the algorithm to a certain standard representation. Therefore, by induction hypothesis, $\Pi'$ possesses a unique irreducible quotient, $\pi'$. It is important to note that the hypothesis also guarantees that $\pi'$ appears in $\Pi'$ with multiplicity one. Now let
\[
\xi =\Delta^\vee \otimes \pi';
\]
note that $\Delta^\vee = \drep{-b}{-a} \times \dotsm \times \drep{-b}{-a}$ ($m$ times). We show that $\xi$ satisfies the required properties. Letting $\pi_1$ be any irreducible quotient of $\Pi$, we may write $\pi_1 \hookrightarrow\Delta^\vee \rtimes  \Pi'^t$,
with $
\Pi'^t = \drep{-b_{r-m}}{-a_{r-m}} \times \dotsm \times \drep{-b_j}{-a_j} \times \zeta(-k,-b_j-1) \times \drep{-b_{j-1}}{-a_{j-1}} \times \dotsm \times \drep{-b_1}{-a_1} \rtimes \tau$. Frobenius reciprocity now implies that there exists a non-trivial intertwining
\[
R_P(\pi_1) \to \Delta^\vee \otimes \Pi'^t.
\]
Note that the fact that $\pi'$ is the unique irreducible quotient of $\Pi'$ translates to the fact that $\pi'$ is the unique irreducible subrepresentation of $\Pi'^t$. Therefore, from Lemma III.3 of \cite{MVW_Howe} it follows that the image of the above map necessarily contains $\Delta^\vee \otimes \pi'$,
i.e.~$\xi$. This proves the first required property of $\xi$.

It remains to show that $\xi$ appears in $R_P(\Pi)$ with multiplicity one. This requires only a simple application of Tadić's formula \eqref{eq_tadic_classical}. Indeed, any $\xi$ which is a subquotient of $R_P(\Pi)$ must occur in
\begin{equation}
\begin{gathered}
\label{eq_semis}
M^*(\drep{a_r}{b_r}) \times M^*(\drep{a_{r-1}}{b_{r-1}}) \times \dotsm \times M^*(\drep{a_j}{b_j})\\
 \times M^*(\zeta(b_j+1,k)) \times M^*(\drep{a_{j-1}}{b_{j-1}}) \times \dotsm \times M^*(\drep{a_1}{b_1}) \rtimes \mu^*(\tau).
\end{gathered}
\end{equation}
For a moment, consider only the $\GL$-part of $\xi$, i.e.~$\Delta^\vee$. We claim that the only way $\Delta^\vee$ can appear in the above formula is if the first $m$ segments participate with $\drep{-b}{-a}\otimes \mathbb{1}$. To prove this, we look at the possible sources of $(-b)$'s in $\Delta^\vee$.

Notice that $\mu^*(\tau)$ cannot add a $(-b)$ to $\Delta^\vee$. Indeed,  the non-degeneracy of $\Delta^\vee$ would imply that $\mu^*(\tau)$ contains a subquotient of the form $\drep{-b}{-a'} \otimes \tau'$ for some $a' \geq a$. However, this contradicts the temperedness of $\tau$ by Casselman's criterion, and is therefore impossible. Taking \eqref{eq_tadic_classical2} into account, we see that the only possible sources of $(-b)$'s are the factors $\drep{a_i}{b_i}$ with ending in $b_i = b$. However, if $\drep{a'}{b}$ is a segment ending in $b$, then the only way $M^*(\drep{a'}{b})$ can add a $(-b)$ to the $\GL$-part is if participates with $\drep{-b}{-a'} \otimes \mathbb{1}$---this follows from formula \eqref{eq_tadic_classical2}. In other words, if $\drep{a'}{b}$ contributes with $(-b)$, then it also contributes with $(-a')$. This shows that no segment $[a',b]$ with $a' \lneqq a$ may be a source of $(-b)$ in $\Delta^\vee$. We deduce that all of the $m$ $(-b)$'s in $\Delta^\vee$ must come from the first $m$ segments. But this also means that the first $m$ segments must participate with $\drep{-b}{-a} \otimes \mathbb{1}$, and the proof of our claim is complete.

This shows that any occurrence of $\xi = \Delta^\vee \otimes \pi'$ must come from $\Delta^\vee \otimes \Pi'$. However, the inductive hypothesis implies $\Pi'$ only contains $\pi'$ with multiplicity one. Therefore, $\xi$ also appears with multiplicity one. This proves the second property of $\xi$ and completes the inductive step in Case 1.

We now turn to Case 2. The only reason the above proof does not work in this case is that now $k=b$, so $\zrep{b_j+1}{b}$ can also be a source of $(-b)$ in $\Delta^\vee$. We therefore modify our approach to avoid this problem.

Let us consider Case 2.1. In this case, we have $b_r = k$, that is, the leftmost segment ends in $k$. We know that the first segment changed by the algorithm was $[b_0,k]$; after the first iteration, we are left with $[b_0,k-1]$. The assumption in Case 2.1 is that this segment is non-empty. Therefore, at least one of the segments $[a_j,b_j],\dotsc, [a_{r-1},b_{r-1}]$ is equal to $[b_0,k-1]$. Now let $[c,k-1]$ denote the shortest segment ending in $k-1$ which appears among $[a_i,b_i]$'s. There may be segments ending in $k$, but none of them are shorter than $[c,k]$---this would contradict the initial choice of $[b_0,k]$ since $c \geq b_0$ (in fact, $[c,k]$ can only appear if $c = b_0$). Therefore, $\Pi$ has the following form:
\begin{align*}
{\drep{c}{k}\times \dotsb \times \drep{c}{k}} &\times (\text{longer segments} \text{ ending in }k) \\
&\times {\drep{c}{k-1}\times \dotsb \times \drep{c}{k-1}} \times \dotsm \rtimes \tau.
\end{align*}
Here $\drep{c}{k}$ appears $m_1$ times (we allow $m_1$ to be zero), and $\drep{c}{k-1}$ appears $m_2>0$ times. We set $m=m_1+m_2$. Using the fact that $\drep{c}{k-1}$ is not linked with $\drep{c}{k}$ or any longer segment ending in $k$, and that $\drep{c}{k-1}\times|\cdot|^k \twoheadrightarrow \drep{c}{k}$, we may write $\Pi$ as a quotient of the following representation:
\[
\drep{c}{k-1}\times \dotsm \times \drep{c}{k-1} \times |\cdot|^k \times \dotsm \times |\cdot|^k \times (\text{longer segments ending in }k) \times \dotsm \rtimes \tau.
\]
Here $\drep{c}{k-1}$ appears $m$ times, and $|\cdot|^k$ appears $m_1$ times. In short, $\Pi$ is a quotient of $\Delta \times (|\cdot|^k,m_1) \times \Pi'$, where 
\begin{align*}
\Delta &= \drep{c}{k-1}\times \dotsm \times \drep{c}{k-1} \quad (m \text{ times})\\
(|\cdot|^k,m_1) &= |\cdot|^k \times \dotsm \times |\cdot|^k \quad (m_1 \text{ times})
\end{align*}
and $\Pi' = \widehat{\drep{a_r}{b_r}} \times \dotsm \widehat{\drep{a_j}{b_j}} \times \zrep{b_j+1}{k} \times {\drep{a_{j-1}}{b_{j-1}}} \times \dotsm \times \drep{a_1}{b_1} \rtimes \tau$. Here $\widehat{\drep{a_i}{b_i}}$ denotes that we are omitting $\drep{a_i}{b_i}$ if $[a_i,b_i] = [c,k-1]$ or $[c,k]$. In particular,
\[
\begin{gathered}\label{eq_segmenti1}\tag{\textasteriskcentered}
\text{any segment appearing in } \Pi' \text{ which contains } k-1\\
\text{necessarily begins with a number strictly smaller than } c.
\end{gathered}
\]
We may now proceed with the proof just like in Case 1. By Remark \ref{rem_alg} (v), $(|\cdot|^k,m_1)\rtimes \Pi'$ is the output of several iterations of the algorithm applied to some standard module---if $c > b_0$, this standard module is obtained from the standard module of $\pi$ by removing any factors of the form $\drep{c}{k-1}$; if $c=b_0$, we remove all the factors of the form or $\drep{b_0}{k}$ or $\drep{b_0}{k-1}$ and replace them with exactly $m_1+1$ factors $|\cdot|^k$---we leave the verification of this fact to the reader. Moreover, $(|\cdot|^k,m_1)\rtimes \Pi'$ is induced from a strictly smaller number of factors than $\Pi$: we have removed exactly $m_2$ of them. Therefore, the induction hypothesis implies that $(|\cdot|^k,m_1)\rtimes \Pi'$ possesses a unique irreducible quotient---which we call $\pi'$---and that this quotient appears with multiplicity one. Set $\xi = \Delta^\vee \otimes \pi'$. Just like in Case 1, Frobenius reciprocity now easily implies that if $\pi_1$ is any irreducible quotient of $\Delta\times (|\cdot|^k,m_1)\rtimes \Pi'$, then $R_P(\pi_1)$ contains $\xi$.

It remains to show that $\xi$ appears in $R_P(\Pi)$ with multiplicity one. To this end, we employ the strategy from Case 1. First, we rewrite equation \eqref{eq_semis} starting with $\Delta \times (|\cdot|^k,m_1) \times \Pi' \twoheadrightarrow \Pi$: any subquotient $\xi$ of $R_P(\Pi)$ has to appear in
\[
\begin{gathered}
M^*(\drep{c}{k-1}) \times \dotsm \times M^*(\drep{c}{k-1})
\times M^*(|\cdot|^{k}) \times \dotsm \times M^*(|\cdot|^{k})\\
\times M^*(\widehat{\drep{a_r}{b_r}}) \times \dotsm M^*(\widehat{\drep{a_j}{b_j}}) \times M^*(\zrep{b_j+1}{k}) \\
\times M^*({\drep{a_{j-1}}{b_{j-1}}}) \times \dotsm \times M^*(\drep{a_1}{b_1}) \rtimes \mu^*(\tau).
\end{gathered}
\]
We now show that the only way $\Delta^\vee$ can appear as the $\GL$-factor in $R_P(\Pi)$ is if all the $\drep{c}{k-1}$'s participate with $\drep{1-k}{-c} \otimes \mathbb{1}$.

We do this by checking the possible sources of $1-k$. First, $\mu^*(\tau)$ cannot be add $1-k$ to $\Delta^\vee$ because of Casselman's criterion, just like in Case 1. Next, because of \eqref{eq_segmenti1}, any segment from $\Pi'$ which could add a $(1-k)$ to $\Delta^\vee$ would also contribute with a number $c' > -c$. Since no such $c'$ appears in $\Delta^\vee$, we deduce that representations $\drep{a_i}{b_i}$ from $\Pi'$ do not contribute to $\Delta^\vee$. The only possible source remaining is $\zrep{b_j+1}{k}$. However, formula \eqref{eq_tadic_classical2} shows that if $M^*(\zrep{b_j+1}{k})$ adds $1-k$ to $\Delta^\vee$, it must also add $(-k)$, and this is impossible since $\Delta^\vee$ does not contain $(-k)$.

This shows that any occurrence of $\xi = \Delta^\vee \otimes \pi'$ must come from $\Delta^\vee \otimes (|\cdot|^k,m_1) \times \Pi'$. Furthermore, just like in Case 1, the inductive hypothesis implies $(|\cdot|^k,m_1) \times \Pi'$ contains $\pi'$ with multiplicity one, so $\xi$ also appears with multiplicity one. We have thus completed the inductive step in Case 2.1.

Finally, we turn to Case 2.2. Recall that we now have $b_0 = k$. Our choice of $\xi$ is slightly more complicated in this case. Let $\pi_1$ be an irreducible quotient of $\Pi$. Also, let $m\geq 1$ be the total number of times the segment $[k,k]$ appears in the original standard module. Recall that $\Pi$ equals
\[
\drep{a_r}{b_r} \times \drep{a_{r-1}}{b_{r-1}} \times \dotsm \times \drep{a_j}{b_j} \times \zeta(b_j+1,k) \times \drep{a_{j-1}}{b_{j-1}} \times \dotsm \times \drep{a_1}{b_1} \rtimes \tau.
\]
By our assumption, $a_r = b_r = \dotsb = a_{r-m+2} = b_{r-m+2} = k$. We now concentrate on $\drep{a_j}{b_j} \times \zeta(b_j+1,k)$. By Lemma \ref{lem_zel1}, this representation has two irreducible subquotients: a unique irreducible subrepresentation which we denote by $s$, and a unique irreducible quotient, denoted $q$. We still do not now which one of them participates in the epimorphism $\Pi \twoheadrightarrow \pi_1$, so we have two options:
\begin{gather}
Q = \drep{a_r}{b_r} \times \dotsm \times q \times \drep{a_{j-1}}{b_{j-1}} \times \dotsm \times \drep{a_1}{b_1} \rtimes \tau \twoheadrightarrow \pi_1\tag{Q}\\
S= \drep{a_r}{b_r} \times \dotsm \times s \times \drep{a_{j-1}}{b_{j-1}} \times \dotsm \times \drep{a_1}{b_1} \rtimes \tau \twoheadrightarrow \pi_1 \tag{S}
\end{gather}
Recall that $q$ can also be written as the quotient of $\zrep{b_j+2}{k} \times \drep{a_j}{b_{j}+1}$. By Lemma \ref{lem_zel2} (ii) and (i), $q$ can switch places with $\drep{a_j}{b_j}$ as well as any other segments (appearing to the left) of the form
\begin{itemize}
\item $[a',b_j]$, with $a' \geq a_j$
\item $[a',b_j+1]$, with $a' \leq a_j$.
\end{itemize}
After pushing $q$ to the left by performing these switches, our inductive hypothesis shows that $Q$ has a unique irreducible quotient; in fact, going from $\Pi \twoheadrightarrow \pi_1$ to $Q \twoheadrightarrow \pi_1$ amounts to taking one step back in our algorithm. Observe that in this case $\pi_1$ is the quotient of the original standard module, in which the segment $[k,k]$ appears $m$ times.

On the other hand, recall that $s$ is the (unique) irreducible quotient of $\zeta(b_j+1,k) \times \drep{a_j}{b_j}$. By Lemma \ref{lem_zel2} (ii), $s$ may switch places with any of the segments (appearing to the left) ending in $b_j$, as they are all shorter than $[a_j,b_j]$. After this, the induction hypothesis again shows that  $S$ possesses a unique irreducible quotient. Thus, if $S \twoheadrightarrow \pi_1$, then $\pi_1$ is the unique irreducible quotient of $S$. Note that this time, running the algorithm backwards will not lead to the original standard module. However, it is easy to see that the standard module of $\pi_1$ in this case will also contain the segment $[k,k]$ exactly $m$ times.

In fact, the reason for the above discussion is that we have now explained the following:
\[
\begin{gathered}\label{eq_mtimes}\tag{$\dagger$}
\text{if }\Pi \twoheadrightarrow \pi_1,\text{ then the standard module of }\pi_1\\\text{contains the segment }[k,k]\text{ precisely }m\text{ times}.
\end{gathered}
\]
Continuing our discussion, we let $\pi_1$ be a quotient of $\Pi$. Instead of working with $q$ and $s$, we now simply observe that (whenever $b_j+1 \lneqq k$; in particular, in Case 2.2)
\[
|\cdot|^k \times \drep{a_j}{b_j} \times \zeta(b_j+1,k-1) \twoheadrightarrow \drep{a_j}{b_j} \times \zeta(b_j+1,k).
\]
Thus, from $\Pi \twoheadrightarrow \pi_1$ we get
\[
\drep{a_r}{b_r} \times \dotsm \times |\cdot|^k \times \drep{a_j}{b_j} \times \zeta(b_j+1,k-1) \times \dotsm \times \drep{a_1}{b_1} \rtimes \tau \twoheadrightarrow \pi_1.
\]
We now push the representation $|\cdot|^k$, i.e.~the segment $[k,k]$, to the left. It may freely switch places with any segments (to the left) not ending in $k-1$. On the other hand, when switching places with a factor of the form $\drep{a'}{k-1}$, we have two options: either $L(|\cdot|^k \rtimes \drep{a'}{k-1})$ or $\drep{a'}{k}$ participates in the epimorphism onto $\pi_1$. However, \eqref{eq_mtimes} shows that $\drep{a'}{k}$ cannot participate, as this would easily lead to $\pi_1$ having a standard module with only $m-1$ segments $[k,k]$. Taking into account that $L(|\cdot|^k \rtimes \drep{a'}{k-1})$ is a quotient of $|\cdot|^k \rtimes \drep{a'}{k-1}$, this means that $|\cdot|^k$ may actually switch places with $\drep{a'}{k-1}$. This discussion shows that we may write
\[
(|\cdot|^k,m) \rtimes \Pi' \twoheadrightarrow \pi_1,
\]
where $\Pi'$ is obtained from $\Pi$ by removing all segments of the form $[k,k]$  which appear to the left of $\zeta(b_j+1,k)$, and truncating $\zrep{b_j+1}{k}$ to $\zrep{b_j+1}{k-1}$. Again, we point out that $\Pi'$ is a representation which can be obtained by applying our algorithm---this time with input $(k-1)$---to some standard module. (This standard module is obtained from the original standard module by removing all factors of the form $|\cdot|^k$). Therefore, the induction hypothesis shows that $\Pi'$ possesses a unique irreducible quotient, which we denote by $\pi'$. The rest of the proof now proceeds as before. We set
\[
\xi = (|\cdot|^{-k},m) \otimes \pi'.
\]
A Jacquet module computation similar to those from previous cases shows that $\xi$ appears in $R_P(\Pi)$ with multiplicity one, and that it appears in the Jacquet module of any quotient of $\Pi$. This completes the inductive step in the last of our cases, and thus concludes the proof of Proposition \ref{prop_unique_q}.
\end{proof}

\begin{exmp}
\label{ex_ind}
We illustrate our choice of $\Pi'$ in the proof of Proposition \ref{prop_unique_q} by giving another concrete example. Consider the representation
\[
\drep{4}{5} \times \drep{2}{4} \times |\cdot|^3 \times \zrep{4}{5} \rtimes \tau.
\]
Using the same notation as in the proof of Proposition \ref{prop_unique_q}, in the above representation we have $k=5$ and $b_r = 5$. Furthermore, as one may easily check, this representation is obtained by applying the algorithm to the standard representation $|\cdot|^5 \times \drep{4}{5}\times \drep{3}{4} \times \drep{2}{4} \rtimes \tau$. Therefore, Case 2.2 applies, so we get $\Pi'$ by shortening $\zrep{4}{5}$ to $|\cdot|^4$:
\[
\drep{4}{5} \times \drep{2}{4} \times |\cdot|^3 \times |\cdot|^4 \rtimes \tau
\]
Next, we show that the above representation also has a unique irreducible quotient. We have $k=4$ and $b_r = 5$, so we are in Case 1. Thus
\[
\Delta = \drep{4}{5} \quad \text{and} \quad \Pi' = \drep{2}{4} \times |\cdot|^3 \times |\cdot|^4\rtimes \tau.
\]
We have thus reduced our representation to $\drep{2}{4} \times |\cdot|^3 \times |\cdot|^4\rtimes \tau$. Here $k=b_r=4$ and the initial segment is easily seen to be $[3,4]$, so Case 2.1 applies. We set
\[
\Delta = |\cdot|^3 \quad \text{and} \quad \Pi' = \drep{2}{4} \times |\cdot|^4\rtimes \tau.
\]
Finally, the representation $\drep{2}{4} \times |\cdot|^4 \rtimes \tau$ is also (trivially) treated by Case 2.2. We thus arrive at the standard representation $\drep{2}{4} \rtimes \tau$, which is the base case.
\end{exmp}

Before applying the algorithm to analyze theta correspondence, we prove another technical lemma (and provide an example) which we use in Section \ref{sec_lifts}. The reader is advised to skim through this result until we invoke it later in the proof of Theorem \ref{theorem_appearance}. For the purposes of the following lemma, we introduce the following notation: For any $\pi \in \Irr(G_n)$ and $k\in \alpha + \mathbb{Z}$, we let $\text{len}_k(\pi)$ denote the number of iterations which the algorithm with input $k$---applied to the $\alpha$-block of $\pi$---performs before exiting the loop. In other words, $\text{len}_k(\pi)$ is the length of the ladder described in Remark 3.5 (iii).
\begin{lem}
\label{lem_len_k}
Consider the representation \eqref{eq_line9} (in the proof of Proposition \ref{prop_unique_q}) obtained by applying several (not necessarily all) iterations of the algorithm to $\pi \in \Irr(G_n)$:
\[
\drep{a_r}{b_r} \times \drep{a_{r-1}}{b_{r-1}} \times \dotsm \times \drep{a_j}{b_j} \times \zeta(b_j+1,k) \times \drep{a_{j-1}}{b_{j-1}} \times \dotsm \times \drep{a_1}{b_1} \rtimes \tau.
\]
Denote this representation by $\textnormal{alg}(\pi)$. We now apply a certain transformation to $\textnormal{alg}(\pi)$: choose $k\geq b' \geq b_j$ and let $p \in \{j,\dotsc,r\}$ be the unique index such that $b_p = b'$ and such that $[a_p,b_p] = [a_p,b']$ is one of the segments which were altered by the algorithm. For any $j' \leq p$, consider the representation $T(k-b',p-j')(\textnormal{alg}(\pi))$ given by
\begin{equation}
\label{eq_len_k}
\drep{a_r}{b_r} \times \drep{a_{r-1}}{b_{r-1}} \times \dotsm \times \drep{a_{j'}}{b_{j'}} \times \zeta(b'+1,k) \times \drep{a_{j'-1}}{b_{j'-1}} \times \dotsm \times \drep{a_1}{b_1} \rtimes \tau.
\end{equation}
In other words, we transform $\textnormal{alg}(\pi)$ by omitting some initial part of the segment which defines $\zrep{b_j+1}{k}$ and inserting the remaining part $\zrep{b'+1}{k}$ anywhere to the right of $[a_p,b_p]$. Then any irreducible quotient $\pi'$ of \eqref{eq_len_k} satisfies $\textnormal{len}_k(\pi') \leq \textnormal{len}_k(\pi)$.
\end{lem}

\begin{proof}
We prove the Lemma for any $\pi \in \Irr(G_n)$ by induction on $p-j'$ and $k-b'$.
Let $[c_1,d_1],\dotsc,[c_t,d_t]$ be the longest ladder with top rung $[c_t,d_t] = [b_0,k]$, as described in Remark \ref{rem_alg} (iii) (here $t = \text{len}_k(\pi)$). First, assume $p-j'= 0$. We claim that
\begin{equation}
\label{eq_len_2}
\drep{a_r}{b_r} \times \dotsm \times \drep{a_p}{b_p} \times \zeta(b'+1,k) \times \drep{a_{p-1}}{b_{p-1}} \times \dotsm \times \drep{a_1}{b_1} \rtimes \tau,
\end{equation}
i.e.~$T(k-b_j',0)(\text{alg}(\pi))$, has a unique irreducible quotient. Indeed, the segments $[a_p,b_p], \dotsc,\allowbreak [a_r,b_r]$ are precisely the same as those in $\text{alg}(\pi)$. Therefore, the above representation is obtained by applying $k-b'$ steps of the algorithm to the $\alpha$-block of some representation $\pi'$. Proposition \ref{prop_unique_q} now shows that $\pi'$ is the unique irreducible quotient of \eqref{eq_len_2}. In fact, the only difference between $\pi$ and $\pi'$ is in the segments which belong to the ladder $[c_t,d_t],\dotsc,[c_1,d_1]$: in place of $[c_i,d_i]$, $\pi'$ has $[c_i,d_i-1]$, for all indices $i = t-k+b_j+1,\allowbreak \dotsc, t-k+b'$ (in other words, the rungs ending in $b_j+1,\dotsc,b'$ are shortened).

Having established the difference between $\pi'$ and $\pi$, we now prove $\text{len}_k(\pi') \leq \text{len}_k(\pi)$. Recall that $t = \text{len}_k(\pi)$. If no ladder of length $t$ exists in the $\alpha$--block of $\pi'$, we are done. 
Otherwise, let $[c_1',d_1'],\dotsc, [c_t',d_t']$ be the ladder of length $t$ with $d_t' = k$ which minimizes segment lengths, in the sense explained in Remark \ref{rem_alg} (iii). We claim that $d_1'-c_1' \geq d_1-c_1$. We prove that $d_i'-c_i' \geq d_i - c_i$ for all $i=t,t-1,\dotsc,1$, by induction. In fact, the discussion in the preceding paragraph explains that $[c_i,d_i] = [c_i',d_i']$ for $i =t, \dotsc, t-k+b'+1$. Now assume that $d_i'-c_i' \geq d_i - c_i$ for some index $i = 2,\dotsc, t$. By construction, $[c_{i-1},d_{i-1}]$ is the shortest segment in the $\alpha$-block of $\pi$ which ends in $d_{i-1}$ and is linked to $[c_i,d_i]$. Since $[c_i',d_i']$ is longer than $[c_i,d_i]$, the only way $[c_{i-1}',d_{i-1}']$ can be shorter than $[c_{i-1},d_{i-1}]$ (and still be linked to $[c_i',d_i']$) is if $[c_{i-1}',d_{i-1}']$ is one of the segments in which $\pi'$ differs from $\pi$, i.e.~one of the segments altered by the algorithm. This would imply $[c_{i-1}',d_{i-1}'] = [c_i,d_i-1]$; however, in that case $[c_i,d_i-1]$ is contained by $[c_i',d_i']$, so it cannot form the next rung.

This completes the inductive step and proves that $d_1'-c_1' \geq d_1-c_1$. Consequently, the ladder in $\pi'$ is even wider than the one in $\pi$, so it cannot be extended lower by one of the segments from the $\alpha$-block of $\pi$---this would contradict the maximum length property of the chosen ladder. The only other available segment is $[c_1,d_1-1]$, but (just like in the above induction step), this segment is contained in $[c_1',d_1']$, so it cannot extend the ladder.

We have thus completed the base case $p-j'=0$. The other base case, that is, the case when $b' = k$ (so that $[b'+1,k]$ is empty) is treated using a similar inductive argument; we leave the details to the reader.

With the base cases settled, we now turn to the general induction step. The induction hypothesis is that the statement of this lemma is valid for any irreducible representation $\sigma$ of $G_n$ (not just $\pi$) and any representation $T(a,b)(\text{alg}(\sigma))$ for $(a,b)$ such that either $(a,b)=(k-b',p-j'-1)$ or $(a,b)=(k-b'-1,\beta),$ for any $\beta$.

Now let $\pi'$ be an irreducible quotient of \eqref{eq_len_k}. 
We have two cases, depending on which subquotient of $\drep{a_{j'}}{b_{j'}} \times \zeta(b'+1,k)$ participates in the epimorphism onto $\pi'$. If it is the (unique) quotient of $\zeta(b'+1,k)\times \drep{a_{j'}}{b_{j'}}$, then instead of \eqref{eq_len_k} we may write
\[
\drep{a_r}{b_r} \times \dotsm \times \zeta(b'+1,k) \times \drep{a_{j'}}{b_{j'}} \times  \drep{a_{j'-1}}{b_{j'-1}} \times \dotsm \times \drep{a_1}{b_1} \rtimes \tau \twoheadrightarrow \pi',
\]
i.e.~$T(k-b',p-j'-1)(\textnormal{alg}(\pi)) \twoheadrightarrow \pi'$. We have thus decreased $p-j'$ to $p-j'-1$, and the induction step is done. Note that this works whenever $b_{j'} \neq b'$ so that $\drep{a_{j'}}{b_{j'}} \times \zeta(b'+1,k)$ is irreducible.

When $b_{j'} = b'$, we need to consider another subquotient, namely the unique irreducible quotient of $\zeta(b'+2,k) \times \drep{a_{j'}}{b'+1}$, which we temporarily denote by $q$. Now assume that $q$ is the irreducible subquotient of $\drep{a_{j'}}{b_{j'}} \times \zeta(b'+1,k)$ which participates in \eqref{eq_len_k}, so that we have
\[
\drep{a_r}{b_r} \times \dotsm \times q \times  \drep{a_{j'-1}}{b_{j'-1}} \times \dotsm \times \drep{a_1}{b_1} \rtimes \tau \twoheadrightarrow \pi'.
\]
According to Lemma \ref{lem_zel2}, we may now push $q$ to the left by switching its position with any segments of the form $[c,b']$ with $c \geq a_j'$, and also $[a,b'+1]$ with $a \leq a_{j'}$. We let $[a_{j''},b_{j''}]$ denote the leftmost of all such segments. Performing these switches, we get
\[
\drep{a_r}{b_r} \times \dotsm \times q \times  \drep{a_{j''}}{b_{j''}} \times \dotsm \times \drep{a_1}{b_1} \rtimes \tau \twoheadrightarrow \pi'.
\]
Finally, taking into account $\zeta(b'+2,k) \times \drep{a_{j'}}{b'+1} \twoheadrightarrow q$, we obtain
\begin{equation}
\label{eq_len_k2}
\drep{a_r}{b_r} \times \dotsm \times \zeta(b'+2,k) \times \drep{a_{j'}}{b'+1} \times  \drep{a_{j''}}{b_{j''}} \times \dotsm \times \drep{a_1}{b_1} \rtimes \tau \twoheadrightarrow \pi'.
\end{equation}
Notice that this is a representation of the form $T(k-b'-1,\beta)(\text{alg}(\pi''))$, for a certain $\pi''$ we now describe (here $\beta$ denotes the suitable value for the second parameter of $T$ which is unknown, but also irrelevant to our argument). In general, $\pi''$ is not necessarily unique (as the transformations $T \circ \text{alg}$ are not injective). However, notice that, in going from \eqref{eq_len_k} to \eqref{eq_len_k2}, the only changes we made to the segments (apart from reordering) were shortening $\zeta(b'+1,k)$ to $\zeta(b'+2,k)$ and changing $\drep{a_{j'}}{b_{j'}}$ to $\drep{a_{j'}}{b'+1}$. Therefore, we may choose $\pi''$ so that its standard module is obtained by altering exactly two segments in the $\alpha$-block of $\pi$, by transfering the cuspidal representation $|\cdot|^{b'+1}$ from $[a_p,b_p+1]$ to $[a_{j'},b_{j'}]$ (recall that $b_{j'} = b_p = b'$). In other words, the only difference in the the $\alpha$-blocks of $\pi$ and $\pi'$ is
\[
\pi
\begin{cases}
[a_p,b_p+1]\\
[a_{j'},b_{j'}]
\end{cases}
\mapsto
\pi'' \begin{cases}
[a_p,b_p]\\
[a_{j'},b_{j'}+1]
\end{cases}
\]
To summarize, $\pi'$ is a quotient of $T(k-b'-1,\beta)(\text{alg}(\pi))$. Since we have now decreased $k-b'$ to $k-b'-1$, the induction hypothesis applies, so we have $\text{len}_k(\pi') \leq \text{len}_k(\pi'')$. Thus, to complete the inductive step, it is enough to check that $\text{len}_k(\pi'') \leq \text{len}_k(\pi)$. This is done by imitating the inductive argument used in the base cases; we leave this simple verification to the reader.
\end{proof}
To illustrate the induction step in the above lemma, we provide another example.
\begin{exmp}
\label{ex_len_k}
Let $\pi$ be the unique irreducible quotient of the standard module
\[
\drep{5}{6} \times \drep{4}{5} \times \drep{1}{5} \times \drep{3}{4} \times \drep{2}{4} \rtimes \tau.
\]
Running the algorithm  with $k=6$ gives us $\text{alg}(\pi)$:
\[
|\cdot|^5 \times \drep{1}{5} \times |\cdot|^4 \times \drep{2}{4} \times |\cdot|^3 \times\zrep{4}{6} \rtimes \tau.
\]
For the sake of the example, let us now apply $T(2,1)$ to $\text{alg}(\pi)$---notice that by setting $2$ as the first parameter of $T$, we are implicitly also setting $[a_p,b_p] = [4,4]$ (because $|\cdot|^4$ is associated to the unique  segment which ends in $4$ and was altered by the algorithm). Thus $T(2,1)(\text{alg}(\pi))$ equals
\[
|\cdot|^5 \times \drep{1}{5} \times |\cdot|^4 \times \drep{2}{4} \times\zrep{5}{6} \times |\cdot|^3  \rtimes \tau.
\]
We let $\pi'$ be an irreducible quotient of the above representation. To illustrate the inductive step, assume that we are in the non-trivial case, where $q = L(|\cdot|^6 \times \drep{2}{5})$ is the irreducible subquotient of $\drep{2}{4} \times\zrep{5}{6}$ which participates in the epimorphism $T(2,1)(\text{alg}(\pi)) \twoheadrightarrow \pi'$. We then have
\[
|\cdot|^5 \times \drep{1}{5} \times |\cdot|^4 \times q \times |\cdot|^3  \rtimes \tau \twoheadrightarrow \pi'.
\]
As mentioned in the proof, Lemma \ref{lem_zel2} now allows us to push $q$ to the left:
\[
|\cdot|^5 \times q \times \drep{1}{5} \times |\cdot|^4 \times |\cdot|^3  \rtimes \tau \twoheadrightarrow \pi'.
\]
Recalling that $q$ is a quotient of $|\cdot|^6 \times \drep{2}{5}$, we thus get
\[
|\cdot|^5 \times |\cdot|^6 \times \drep{2}{5} \times  \drep{1}{5} \times |\cdot|^4 \times |\cdot|^3  \rtimes \tau \twoheadrightarrow \pi'.
\]
As explained in the proof, this left hand side can be viewed as $T(1,0)(\text{alg}(\pi''))$, where $\pi''$ is the unique irreducible quotient of
\[
\drep{5}{6} \times \drep{2}{5} \times  \drep{1}{5} \times |\cdot|^4 \times |\cdot|^3  \rtimes \tau.
\]
Indeed, applying only one iteration of the algorithm with $k=6$, we obtain $\text{alg}(\pi'')$:
\[
|\cdot|^5 \times |\cdot|^6 \times \drep{2}{5} \times  \drep{1}{5} \times |\cdot|^4 \times |\cdot|^3  \rtimes \tau.
\]
Now $T(1,0)$ does not change anything (because $b'+1 = k = 6$ and $[a_p,b_p]$ is the segment $[5,5]$), so the above representation is already equal to $T(1,0)(\text{alg}(\pi''))$. It remains to compare the standard modules of $\pi$ and $\pi''$. In this case, it is easy to see that (with $k=6$) $\text{len}_k(\pi) = 3$, whereas $\text{len}_k(\pi'') = 2$. Thus $\text{len}_k(\pi'') \leq \text{len}_k(\pi)$, as claimed in the proof.
\end{exmp}

\section{Theta correspondence}
\label{sec_theta}
In this section we review some general results and fix the notation for theta correspondence. We also prove several auxiliary results which we use in subsequent sections.
\subsection{Howe duality}
\label{subs_howe}
Let $\omega_{m,n}$ be the Weil representation of $G(W_n) \times H(V_m)$. The Weil representation depends on the choice of a non-trivial additive character $\psi\colon F\to\C$. This character will be fixed throughout, so we omit it from the notation. Similarly, if the dimensions $m$ and $n$ are known, we will often simply write $\omega$ instead of $\omega_{m,n}$.

For any $\pi \in \Irr(G(W_n))$, the maximal $\pi$-isotypic quotient of $\omega_{m,n}$ is of the form
\[
\pi \otimes \Theta(\pi,V_m)
\]
for a certain smooth representation $\Theta(\pi,V_m)$ of $H(V_m)$, called the full theta lift of $\pi$ (see \cite[Chapter II, III.4]{MVW_Howe}). When the target Witt tower is fixed, we will denote it by $\Theta(\pi,m)$ or, more often, by $\Theta_l(\pi)$, where $l = n + \epsilon - m$ (we recall that $\epsilon$ is defined in 2.1). Note that $l$ is always an integer. Its parity is determined by the dual pair: $l$ is odd when we are working with the symplectic--even orthogonal dual pair, and it is even in case of the metaplectic--odd orthogonal pair. We also let $\kappa \in \{1,2\}$ such that $\kappa \equiv l \pmod 2$.

The following result establishes the theta correspondence:
\begin{thm}[Howe duality]
\label{thm_Howe}
If $\Theta(\pi,V_m)$ is non-zero, it possesses a unique irreducible quotient, denoted $\theta(\pi,V_m)$.
\end{thm}
\noindent Originally conjectured by Howe in \cite{Howe_theta_series}, this was first proven by Waldspurger \cite{Waldspurger_howe_duality} when the residual characteristic of $F$ is different from $2$, and by Gan and Takeda \cite{Gan_Takeda_proof_of_Howe} in general.
The representation $\theta(\pi,V_m)$ is called the (small) theta lift of $\pi$; like the full lift, it will also be denoted $\theta(\pi,m)$ and $\theta_l(\pi)$. The following simple fact is often useful (see Lemma 1.1 of \cite{Muic_theta_discrete_Israel}):
\begin{lem}
\label{lem_ThetaHom}
For $\pi \in \Irr(G(W_n))$ we have
\[
\Theta(\pi,m)^\vee = \Hom_{G_n}(\omega_{m,n},\pi)_\infty.
\]
\end{lem}
\subsection{First occurrence in towers}
\label{subs_first_occurrence}
A basic fact concerning theta correspondence is expressed by the following proposition (see Propositions 4.1 and 4.3 of \cite{Kudla1}):
\begin{prop}[Tower property]
\label{prop_towers}
Let $\pi$ be an irreducible representation of $G(W_n)$. Fix a Witt tower $\mathcal{V} = (V_m)$.
\begin{enumerate}[(i)]
\item If $\Theta(\pi,V_m) \neq 0$, then $\Theta(\pi,V_{m+2r}) \neq 0$ for all $r\geqslant 0$.
\item For $m$ large enough, we have $\Theta(\pi,V_m) \neq 0$.
\end{enumerate}
\end{prop}
The above proposition implies that we can define, for any Witt tower $\mathcal{V} = (V_m)$,
\[
m_\mathcal{V}(\pi) = \min\{m \geqslant 0: \Theta(\pi,V_m) \neq 0\}.
\]
This number (also denoted $m(\pi)$ when the choice of $\mathcal{V}$ is implicit) is called the first occurrence index of $\pi$. Note that we are using the term “index” here to signify the dimension, although it would be more appropriate to use it for the Witt index of the corresponding space.

We recall the so-called conservation relation. The Witt towers of quadratic spaces can be appropriately organized into pairs, with the towers comprising a pair denoted $\mathcal{V}^+$ and $\mathcal{V}^-$; a complete list of pairs of dual towers can be found in \cite[Chapter V]{Kudla1}. Thus, instead of observing just one target tower, we can simultaneously look at two of them. This way, for each $\pi \in \Irr(G(W_n))$ we get two corresponding first occurrence indices, $m^+(\pi)$ and $m^-(\pi)$.

If $\epsilon = -1$ so that $W_n$ is a quadratic space, we proceed as follows: since $G(W_n)$ is now equal to $O(W_n)$, any $\pi \in \Irr(G(W_n))$ is naturally paired with its twist, $\det \otimes \pi$. This allows us to define
\begin{align*}
m^+(\pi) &= \min\{m(\pi), m(\det\otimes\pi)\},\\
m^-(\pi) &= \max\{m(\pi), m(\det\otimes\pi)\}.
\end{align*}
We are now able to set
\[
m^{\text{down}}(\pi) =  \min\{m^+(\pi), m^-(\pi)\}, \quad m^{\text{up}}(\pi) =  \max\{m^+(\pi), m^-(\pi)\}
\]
regardless of whether $\epsilon = 1$ or $\epsilon = -1$. Note that when $W_n$ is a quadratic space, we have $m^{\text{down}}(\pi) = m^+(\pi) $ and $ m^{\text{up}}(\pi) = m^-(\pi)$. The conservation relation (first conjectured by Kudla and Rallis in \cite{Kudla_Rallis_progress}, completely proven by Sun and Zhu in \cite{Sun_Zhu_conservation}) states that
\[
m^{\text{up}}(\pi) + m^{\text{down}}(\pi) = 2n + 2\epsilon + 2.
\]
The tower in which $m(\pi) = m^{\text{down}}(\pi)$ (resp.~$m^{\text{up}}(\pi)$) is called the going-down (resp.~going-up) tower.
\subsection{Kudla's filtration}
\label{subs_Kudla}
We now review Kudla's filtration of $R_P(\omega)$, the Jacquet module of the Weil representation (see Theorem 2.8 of \cite{Kudla2}). We state it here---formulated as in Theorem 5.1 of \cite{Atobe_Gan}---along with a few useful corollaries.
\begin{thm}[Kudla's filtration]
\label{thm_Kudla}
The Jacquet module $R_{P_k}(\omega_{m,n})$ possesses an $\GL_k(F)\times G(W_{n-2k}) \times H(V_m)$-equivariant filtration
\[
R_{P_k}(\omega_{m, n}) = R^0 \supset R^1 \supset \dotsb \supset R^k \supset R^{k+1} = 0
\]
in which the successive quotients $J^a = R^a/R^{a+1}$ are given by
\[
J^a = \text{Ind}_{P_{k-a,a}\times G_{n-2k}\times Q_a}^{\GL_k(F)\times G_{n-2k}\times H_m}\left({\chi_V|\emph{\text{det}}_{\GL_{k-a}}|^{\lambda_{k-a}}\otimes \Sigma_a \otimes \omega_{m-2a, n-2k}}\right),
\]
where
\begin{itemize}
\item $\lambda_{k-a} = (m-n+k-a-\epsilon)/2$;
\item $P_{k-a,a} =$ standard parabolic subgroup of $\GL_k(F)$ with Levi factor $\GL_{k-a}(F) \times \GL_a(F)$;
\item $\Sigma_a = C_c^\infty(\GL_a(F))$, the space of locally constant compactly supported functions on $\GL_a(F)$. The action of $\GL_a(F) \times \GL_a(F)$ on $\Sigma_a$ is given by
\[
[(g,h)\cdot f](x) = \chi_V(\det(g))\chi_W(\det(h))f(g^{-1}\cdot x\cdot h).
\]
\end{itemize}
If $m-2a$ is less than the dimension of the first (anisotropic) space in $\mathcal{V}$, we put $R^a=J^a=0$.
\end{thm}

We will often use the following proposition derived from the previous theorem (see Corollary 3.2 of \cite{Muic_theta_discrete_Israel} and also Proposition 5.2, \cite{Atobe_Gan}):
\begin{prop}
\label{prop_Muic_isotype}
Let $\delta$ be an irreducible essentially square integrable representation of $\GL_k(F)$ and $\pi_0 \in \Irr (G_{n-2k})$, for some $k>0$. Then the space $\Hom_{\GL_k(F)\times G_{n-2k}}(J^a, \chi_V\delta^\vee\otimes \pi_0)_\infty$, viewed as a representation of $H_m$, is isomorphic to
\[
\begin{cases}
\chi_W^{-1}\delta^\vee \rtimes \Hom_{G_{n-2k}}(\omega_{m-2k,n-2k}, \pi_0)_\infty, &\text{if } a = k, \\
\chi_W^{-1}\text{St}_{k-1}\nu^{\frac{k-l+1}{2}} \rtimes \Hom_{G_{n-2k}}(\omega_{m-2k+2,n-2k}, \pi_0)_\infty,&\text{if } a = k-1 \text{ and } \delta = \text{St}_k\nu^{\frac{l-k}{2}}, \\
0, &\text{otherwise.}
\end{cases}
\]
\end{prop}
\noindent Recall that, in the above proposition, we have $\Hom_G(\omega,\pi)_\infty = \Theta(\pi)^\vee$. Furthermore, $\text{St}_k$ denotes the so-called Steinberg representation of $\GL_k(F)$, the square integrable representation attached to the segment $[|\cdot|^{\frac{1-k}{2}},|\cdot|^{\frac{k-1}{2}}]$. Thus $\text{St}_k\nu^{\frac{l-k}{2}} = \drep{\frac{l+1}{2}-k}{\frac{l-1}{2}}$. We point out that the condition $l>0$ given in \cite[Proposition 5.2]{Atobe_Gan} is not necessary (cf.~\cite[Corollary 3.2]{Muic_theta_discrete_Israel}).

We now give two useful corollaries of Proposition \ref{prop_Muic_isotype}. The first one is Corollary 5.3 of \cite{Atobe_Gan}; see also \cite[Corollary 3.2]{Muic_theta_discrete_Israel}.
\begin{cor}
\label{cor_theta_epi}
Let $\pi\in \Irr(G_n), \pi_0 \in \Irr(G_{n-2k})$ and let $\delta$ be an irreducible essentially square integrable representation of $\GL_k(F)$. Assume that $\delta \ncong \textnormal{St}_k\nu^{\frac{l-k}{2}}$, where $l = n - m + \epsilon$. Then
\[
\chi_V\delta \rtimes \pi_0 \twoheadrightarrow \pi
\]
implies
\[
\chi_W\delta \rtimes \Theta_l(\pi_0) \twoheadrightarrow \Theta_l(\pi).
\]
\end{cor}
The second corollary we state is a slight modification of the first: this time, we are unable to obtain information about the full lift $\Theta_l(\pi)$, but we allow the special case $\delta \cong \text{St}_k\nu^{\frac{l-k}{2}}$.
\begin{cor}
\label{cor_shaving}
Let $\delta \in \Irr(\GL_k(F))$ be an essentially square integrable representation and let $\pi\in \Irr(G_n), \pi_0 \in \Irr(G_{n-2k})$ be such that
\[
\chi_V\delta \rtimes \pi_0 \twoheadrightarrow \pi.
\]
Furthermore, let $A$ be a representation of a general linear group. Assume that an irreducible representation $\sigma$ satisfies
\[
A \rtimes \Theta_l(\pi) \twoheadrightarrow \sigma,
\]
where $l = n - m + \epsilon$. If $\delta \ncong \textnormal{St}_k\nu^{\frac{l-k}{2}}$, then
\[
\label{eq_shaving1}\tag{A}
A \times \chi_W\delta \rtimes \Theta_l(\pi_0) \twoheadrightarrow \sigma.
\]
If $\delta \cong \textnormal{St}_k\nu^{\frac{l-k}{2}} \cong \drep{\frac{l+1}{2}-k}{\frac{l-1}{2}}$, then either (A) is true, or the following holds:
\[
\label{eq_shaving2}\tag{B}
A \times \chi_W\drep{\frac{l+1}{2}-k}{\frac{l-3}{2}} \rtimes \Theta_{l-2}(\pi_0) \twoheadrightarrow \sigma.
\]
\end{cor}

\begin{proof}
By assumption, we have $\pi \hookrightarrow \chi_V\delta^\vee \rtimes \pi_0$, so
\begin{align*}
\Theta_l(\pi)^\vee &\cong \Hom_{G_n}(\omega_{m,n},\pi)_{\infty}\\
&\hookrightarrow  \Hom_{G_n}(\omega_{m,n},\chi_V\delta^\vee \rtimes \pi_0)_{\infty}\\
&\cong \Hom_{\GL_k \times G_{n-2k}}(R_{P_k}(\omega_{m,n}),\chi_V\delta^\vee \otimes \pi_0)_{\infty}.
\end{align*}
We now use Kudla's filtration to analyze $R_{P_k}(\omega_{m,n})$. For each index $a = 0, \dotsc, k$ we have an exact sequence
\[
0 \to R^{a+1 }\to R^{a} \to J^{a} \to 0.
\]
Applying the $\Hom(\ \cdot\ , \delta^\vee \otimes \pi_0)_\infty$ functor we get
\begin{equation}
\label{eq_exact}
0 \to \Hom(J^{a},\chi_V\delta^\vee \otimes \pi_0)_{\infty} \to \Hom(R^{a}, \chi_V\delta^\vee \otimes \pi_0)_{\infty} \to \Hom(R^{a+1}, \chi_V\delta^\vee \otimes \pi_0)_{\infty}.
\end{equation}
Since we know, by Proposition \ref{prop_Muic_isotype}, that the space $\Hom(J^{a}, \chi_V\delta^\vee \otimes \pi_0)_{\infty}$ is trivial for $a = 0, \dotsc, k-2$, this leads to an inclusion
\[
\Hom_{\GL_k \times G_{n-2k}}(R_{P_k}(\omega_{m,n}),\chi_V\delta^\vee \otimes \pi_0)_{\infty} \hookrightarrow \Hom_{\GL_k \times G_{n-2k}}(R^{k-1},\chi_V\delta^\vee \otimes \pi_0)_{\infty}.
\]
In particular, we have $\Theta_l(\pi)^\vee \hookrightarrow \Hom_{\GL_k \times G_{n-2k}}(R^{k-1}, \chi_V\delta^\vee \otimes \pi_0)_{\infty}$. Inducing with $A^\vee$, we get 
\[
A^\vee \rtimes \Theta_l(\pi)^\vee \hookrightarrow A^\vee \rtimes \Hom_{\GL_k \times G_{n-2k}}(R^{k-1}, \chi_V\delta^\vee \otimes \pi_0)_{\infty}.
\]
By assumption, we have $\sigma^\vee {\hookrightarrow} A^\vee \rtimes \Theta_l(\pi)^\vee$, so there is an injective equivariant map
\[
f\colon \sigma^\vee {\hookrightarrow} A^\vee\rtimes\Hom_{\GL_k \times G_{n-2k}}(R^{k-1}, \chi_V\delta^\vee \otimes \pi_0)_{\infty}.
\]
On the other hand, we may set $a=k-1$ in \eqref{eq_exact} and induce to get
\begin{align*}
0 \to A^\vee\rtimes\Hom(J^{k-1}&, \chi_V\delta^\vee \otimes \pi_0)_{\infty}\\
&\stackrel{g}{\to} A^\vee\rtimes\Hom(R^{k-1}, \chi_V\delta^\vee \otimes \pi_0)_{\infty} \stackrel{h}{\to} A^\vee\rtimes\Hom(J^{k}, \chi_V\delta^\vee \otimes \pi_0)_{\infty}.
\end{align*}
We now consider two options:
\begin{enumerate}[(A)]
\item If $\text{Im}(f) \cap \text{Ker}(h) = 0$, then we have an injective map
\[
h\circ f\colon \sigma^\vee \hookrightarrow A^\vee\rtimes\Hom(J^{k}, \chi_V\delta^\vee \otimes \pi_0)_{\infty}.
\]
Proposition \ref{prop_Muic_isotype} describes $\Hom(J^{k}, \chi_V\delta^\vee \otimes \pi_0)_{\infty}$; by taking the contragredient we get
\[
A \rtimes \chi_W\delta \rtimes \Theta_l(\pi_0) \twoheadrightarrow \sigma.
\]
Note that $\delta \ncong  \text{St}_k\nu^{\frac{l-k}{2}}$ implies $\Hom(J^{k-1}, \chi_V\delta^\vee \otimes \pi_0)_{\infty}=0$. In that case, $ \text{Ker}(h) = \text{Im}(g) = 0$, so we always have the above result. If $\delta \cong  \text{St}_k\nu^{\frac{l-k}{2}}$, we may have $\text{Im}(f) \cap \text{Ker}(h) \neq 0$.

\item If $\text{Im}(f) \cap \text{Ker}(h) \neq 0$, then the irreducibility of $\sigma$ implies $\sigma^\vee \hookrightarrow  \text{Ker}(h)$. By the exactness of the above sequence we have $\text{Ker}(h) = \text{Im}(g)$, and since $g$ is injective, we also have $ \text{Im}(g) \cong A^\vee \rtimes \Hom(J^{k-1}, \chi_V\delta^\vee \otimes \pi_0)_{\infty}$. Thus, we can write
\[
\sigma^\vee \hookrightarrow \Hom(J^{k-1}, \chi_V\delta^\vee \otimes \pi_0)_{\infty}
\]
from which, by looking at the contragradient (and using Proposition \ref{prop_Muic_isotype}), we arrive at
\[
A \rtimes \chi_W\drep{\frac{l+1}{2}-k}{\frac{l-3}{2}} \rtimes \Theta_{l-2}(\pi_0) \twoheadrightarrow \sigma.
\]
\end{enumerate}
\end{proof}
\section{Occurrence}
In this section, we determine the first occurrence index of $\pi \in \Irr(G_n)$. We fix a pair $\{\mathcal{V}^+,\mathcal{V}^-\}$ of Witt towers. Recall that, for a fixed Witt tower $\mathcal{V} = (V_m)$, $\Theta_l(\pi)$ denotes the theta lift of $\pi$ to $V_m$, with $l = n + \epsilon - m$. To signify the first occurrence index relative to the rank of group $G_n$, we set
\[
l(\pi) = n + \epsilon - m^{\text{down}}(\pi).
\]
following \cite{Atobe_Gan} (note that this definition differs from the one in \cite{Atobe_Gan}, but they coincide when $\pi$ is tempered).
Thus $\pi$ first appears on the going-down tower when $l=l(\pi)$, whereas (by the conservation relation) $l=-l(\pi)-2$ for the first occurrence of $\pi$ on the going-up tower.

We first treat one special case pertaining to the symplectic--even orthogonal dual pair.
\begin{prop}
\label{prop_-1}
Let
\[
\chi_V\delta_r \times \chi_V\delta_{r-1} \times \dotsm \times \chi_V\delta_1 \rtimes \tau
\]
denote the standard module of $\pi$. Assume that $l(\tau) = -1$. Then $l(\pi) = -1$.
\end{prop}
\begin{proof}
It is enough to prove that $\Theta_1(\pi) = 0$ on the going-down tower. Repeatedly applying Corollary \ref{cor_theta_epi} to the epimorphism
\[
\chi_V\delta_r \times \chi_V\delta_{r-1} \times \dotsm \times \chi_V\delta_1 \rtimes \tau \twoheadrightarrow \pi
\]
we arrive at
\[
\chi_W\delta_r \times \chi_W\delta_{r-1} \times \dotsm \times \chi_W\delta_1 \rtimes \Theta_1(\tau) \twoheadrightarrow \Theta_1(\pi).
\]
Notice that none of the segments which define $\delta_i$ end in $\frac{l-1}{2} = 0$, so that condition $\delta \ncong \textnormal{St}_k\nu^{\frac{l-k}{2}}$ of Corollary \ref{cor_theta_epi} is fulfilled. Since $l(\tau)=-1$, we have $\Theta_1(\tau)=0$, so the above epimorphism implies $\Theta_1(\pi) = 0$.
\end{proof}
The following two lemmas will be used in the main proof. They also provide a paradigm of how $l(\pi)$ can be greater than $l(\tau)$, something that cannot happen if, say, $\pi$ is generic (cf.~\cite{Bakic_Hanzer_generic}).
\begin{lem}
\label{lemma_reverse_temp_thin}
Let $\tau \in \Irr(G_n)$ be a tempered representation with $l(\tau) = l \geq 0$. Let $a,b \in \mathbb{Z}$ be such that $b > l \geq a \geq 0$ and $a,b \equiv l \pmod 2$ and let
\[
\pi = L(\chi_V|\cdot|^{\frac{b-1}{2}},\chi_V|\cdot|^{\frac{b-3}{2}},\dotsc,\chi_V|\cdot|^{\frac{a+1}{2}}; \tau).
\]
Then $l(\pi) = b$.
\end{lem}
\begin{proof}
We show that $\theta_b(\tau) \neq 0$ and $\theta_{b+2}(\tau) = 0$ on the going-down tower.

\noindent We have $l(\tau) = l$ and $a \leq l$, so $\Theta_a(\tau) \neq 0$. We thus have $\tau = \theta_{-a}(\theta_a(\tau))$. Now Proposition 5.6 of \cite{Atobe_Gan} shows that $\pi = \theta_{-b}(\theta_a(\tau))$. In particular, we have $\theta_b(\pi) = \theta_a(\tau) \neq 0$.

To prove $\theta_{b+2}(\pi)$ we repeatedly apply Corollary \ref{cor_theta_epi} to
\[
\chi_V|\cdot|^{\frac{b-1}{2}} \times \chi_V|\cdot|^{\frac{b-3}{2}} \times \dotsm \times \chi_V|\cdot|^{\frac{a+1}{2}} \rtimes \tau \twoheadrightarrow \pi.
\]
Since no exponent is equal to $\frac{b+1}{2}$, condition $\delta \ncong \textnormal{St}_k\nu^{\frac{l-k}{2}}$ of Corollary \ref{cor_theta_epi} is fulfilled. We get
\[
\chi_W|\cdot|^{\frac{b-1}{2}} \times \chi_W|\cdot|^{\frac{b-3}{2}} \times \dotsm \times \chi_W|\cdot|^{\frac{a+1}{2}} \rtimes \theta_{b+2}(\tau) \twoheadrightarrow \theta_{b+2}(\pi).
\]
Since $\theta_{b+2}(\tau) = 0$, this implies $\theta_{b+2}(\pi) = 0$.
\end{proof}

\begin{lem}
\label{lemma_reverse_temp_thick}
Let $\tau \in \Irr(G_n)$ be a tempered representation with $l(\tau) = l > 0$. Let $a,b \in \mathbb{Z}$ be such that $b > l \geq a \gneqq 0$ and $a,b \equiv l \pmod 2$ and let
\[
\pi = L(\chi_V|\cdot|^{\frac{b-1}{2}},\chi_V|\cdot|^{\frac{b-3}{2}},\dotsc,\chi_V|\cdot|^{\frac{a+3}{2}},\chi_VSt_{a+1}\nu^{\frac{1}{2}}; \tau).
\]
Additionally, assume that $m_{\phi_\tau}(\chi_VS_a)$ is odd. Then $l(\pi) = b$.
\end{lem}

Before the proof, we point out two details. First, in this lemma, we assume $l > 0$. Otherwise, $l=0$ and $l \geq a \geq 0$ would force $a=0$. In that case, $St_{a+1}\nu^{\frac{1}{2}} = |\cdot|^\frac{a+1}{2}$, which is already treated in Lemma \ref{lemma_reverse_temp_thin}. Secondly, note that since $\Theta_{l}(\tau) \neq 0$ and $a \leq l$, Theorem 4.1 of \cite{Atobe_Gan} shows that the condition imposed on $m_{\phi_\tau}(\chi_VS_a)$ is almost always valid. The only situation in which $m_{\phi_\tau}(\chi_VS_a)$ can be even is if $a=l$; this is treated in Lemma \ref{lemma_fat_even}
\begin{proof}[Proof of Lemma \ref{lemma_reverse_temp_thick}]
First assume that $m_{\phi_\tau}(\chi_VS_a)=1$, so that $\theta_a(\tau)$ does not have $\chi_WS_a$ in its parameter. Then, $\chi_W\delta({\frac{1-a}{2}}, {\frac{a-1}{2}})\rtimes \theta_a(\tau)$ is completely reducible, of length two; let $T$ be one of its two subquotients, the one for which $\eta_{T}(\chi_WS_a) =-\eta_{T}(\chi_WS_{a-2})$(the other subquotient is $\theta_a(\chi_VSt_a\rtimes \tau)$). Then, $l(T)=a$ and $\theta_{-b}(T)= \pi$ by \cite{Atobe_Gan}, Theorem 4.5 (2),(3).

If $m_{\phi}(\chi_VS_a)=2h+1$, then $\tau\cong \chi_V(St_a,h)\rtimes \tau_0$, where now $\tau_0$ satisfies properties from the first part of the proof and $(St_a,h)$ denotes the product of $h$ factors $St_a$. Let $T_0$ be the tempered representation associated to $\tau_0$ analogous to the representation $T$ from the first part of the proof.
Then, the representation $T_1=\chi_V(St_t,h-1)\rtimes T_0$ is irreducible and $l(T_1)=a$ since $l(T_0)=a$. Again using \cite{Atobe_Gan}, Theorem 4.5 (3), since the multiplicity of $S_a$ in $\phi_{T_1}$ is even, we get that $\theta_{-b}(T_1)=\pi$.

This shows $\theta_b(\pi) \neq 0$. The proof that $\theta_{b+2}(\pi) = 0$ is the same as in Lemma \ref{lemma_reverse_temp_thin}.
\end{proof}
Before we state our main result, we need to cover another special case.
\begin{lem}
\label{lemma_fat_even}
Let $\tau \in \Irr(G_n)$ be a tempered representation with $l(\tau) = l > 0$. Furthermore, assume that $m_{\phi_\tau}(\chi_VS_l)$ is even.
Let
\[
\pi = L(\chi_VSt_{l+1}\nu^\frac{1}{2}, \chi_VSt_{l+1}\nu^\frac{1}{2}, \dotsc, \chi_VSt_{l+1}\nu^\frac{1}{2}; \tau).
\]
Then $l(\pi) = l(\tau) = l$.
\end{lem}
\begin{proof}
Let $m_{\phi_\tau}(\chi_VS_l) = 2h > 0$ (the multiplicity must be positive since $\theta_l(\tau) \neq 0)$. Then there exists a tempered representation $\tau_0$ whose parameter does not contain $\chi_VS_l$ such that
\[
\chi_V(St_l,h) \rtimes \tau_0 \twoheadrightarrow \tau
\]
(the representation on the left is completely reducible, of length two).
Denoting by $k$ be the number of factors $\chi_VSt_{l+1}\nu^\frac{1}{2}$ which appear in the standard module of $\pi$, we thus have
\[
\chi_V(St_{l+1}\nu^\frac{1}{2},k) \times \chi_V(St_l,h) \rtimes \tau_0 \twoheadrightarrow \pi.
\]
Since the corresponding segments are not linked, $St_l$ and $St_{l+1}\nu^\frac{1}{2}$ may switch places; we get
\[
\chi_V(St_l,h) \times \chi_V(St_{l+1}\nu^\frac{1}{2},k) \rtimes \tau_0 \twoheadrightarrow \pi.
\]
Repeatedly applying Corollary \ref{cor_theta_epi} we obtain
\[
\chi_W(St_l,h) \rtimes \Theta_{l+2}(\pi_0) \twoheadrightarrow \Theta_{l+2}(\pi),
\]
where we have used $\pi_0$ to denote the unique irreducible quotient of $\chi_V(St_{l+1}\nu^\frac{1}{2},k) \rtimes \tau_0$. We now show that $\Theta_{l+2}(\pi_0)=0$, which implies $\Theta_{l+2}(\pi) = 0$. Note that we cannot apply Corollary \ref{cor_theta_epi} to
\[
\chi_V(St_{l+1}\nu^\frac{1}{2},k) \rtimes \tau_0  \twoheadrightarrow \pi_0
\]
to compute $\Theta_{l+2}(\tau_0)$ because the segment which defines $St_{l+1}\nu^\frac{1}{2}$ ends in $\frac{l+1}{2}$. We therefore use Corollary \ref{cor_shaving}. Repeatedly applying the corollary to the above epimorphism, we find that one of the following must hold:
\begin{itemize}
\item $\chi_W(St_{l+1}\nu^\frac{1}{2},k) \rtimes \Theta_{l+2}(\tau_0) \twoheadrightarrow \theta_{l+2}(\pi_0)$;
\item $\chi_W(St_{l+1}\nu^\frac{1}{2},k-1) \times \chi_WSt_l \rtimes \Theta_{l}(\tau_0) \twoheadrightarrow \theta_{l+2}(\pi_0)$.
\end{itemize}
However, both $\Theta_{l+2}(\tau_0)$ and $\Theta_{l}(\tau_0)$ are equal to zero, since the parameter of $\tau_0$ no longer contains $\chi_VS_l$. This shows $\theta_{l+2}(\pi_0) = 0$, which in turn implies $\theta_{l+2}(\pi) = 0$.

To finish the proof, it remains to show that $\Theta_{l}(\pi) \neq 0$. To do this, we apply Corollary \ref{cor_theta_epi} to 
\[
\chi_V(St_{l+1}\nu^\frac{1}{2},k) \rtimes \tau \twoheadrightarrow \pi,
\]
this time lifting $\pi$ to the going-up tower. We get
\[
\chi_W(St_{l+1}\nu^\frac{1}{2},k) \rtimes \Theta_{-l}(\tau) \twoheadrightarrow \Theta_{-l}(\pi).
\]
Since $\Theta_l(\tau)\neq 0$ on the going-down tower, we must have $\Theta_{-l}(\tau) = 0$ on the going-up tower. The above epimorphism now shows that $\Theta_{-l}(\pi) = 0$ on the going-up tower. Again using the conservation relation, this implies $\Theta_{l}(\pi) \neq 0$ on the going-down tower, as desired.
\end{proof}

We are now ready to prove the main result of this section. For any $\pi \in \Irr(G_n)$, we may write its standard module as
\[
\chi_V\Xi \times \chi_V\delta_r \times \chi_V\delta_{r-1} \times \dotsb \times \chi_V\delta_1 \rtimes \tau.
\]
Here $\chi_V\delta_r \times \chi_V\delta_{r-1} \times \dotsb \times \chi_V\delta_1$ denotes the $\alpha$-block, with $\alpha = \frac{l(\tau)-1}{2}$. The rest of the factors are grouped into $\Xi$. We assume that the $\alpha$-block is sorted as described in \S\ref{subs_algorithm}. The following theorem determines the first occurrence of $\pi$. 
\begin{thm}
\label{theorem_occurrence}
Let
\[
\chi_V\Xi \times \chi_V\delta_r \times \chi_V\delta_{r-1} \times \dotsb \times \chi_V\delta_1 \rtimes \tau
\]
be the standard module of $\pi \in \Irr(G_n)$. Let $\delta_i = \drep{a_i}{b_i}$ for $i=1,\dotsc,r$, and let $l(\tau) = l$.

Consider all subsequences $
[c_1,d_1], [c_2,d_2],\dotsc, [c_k,d_k]$ of the sequence $[a_1,b_1],[a_2,b_2],\dotsc, \allowbreak [a_r,b_r]$ such that
\begin{enumerate}[(i)]
\item $d_{i+1} = d_i + 1$ for $i=1,\dotsc, k-1$;
\item $c_i \lneqq c_{i+1}$,  for $i=1,\dotsc, k-1$;
\item $d_1 = \frac{l+1}{2}$;
\item if $l>0$ and $m_{\phi_\tau}(\chi_VS_l)$ is even, $c_1 \neq \frac{1-l}{2}$.
\end{enumerate}
Let $f$ be the length of the longest subsequence satisfying these conditions. Then $l(\pi) = l + 2f$.
\end{thm}
\begin{proof}
If $l=-1$, then $f=0$ because of (iii)---there is no segment ending in $0$. Therefore, the statement of the theorem in this case amounts to $l(\pi) = -1$; this is proved in Proposition \ref{prop_-1}. Thus we may assume that $l \geq 0$.
First, we prove that $\Theta_{l+2f}(\pi)\neq 0$ on the going-down tower.

Apply the algorithm of \S\ref{subs_algorithm} to the $\alpha$-block of $\pi$ setting the initial value of $k$ to $\frac{l-1}{2}+f$. Let
\begin{equation*}
\begin{gathered}
\chi_V\Xi \times\chi_V\drep{a_r'}{b_r'} \times \chi_V\drep{a_{r-1}'}{b_{r-1}'} \times \dotsm \times \chi_V\drep{a_j'}{b_j'} \times \chi_V\zeta(b_j'+1,k)\\ \times \chi_V\drep{a_{j-1}'}{b_{j-1}'} \times \dotsm \times \chi_V\drep{a_1'}{b_1'} \rtimes \tau
\end{gathered}
\end{equation*}
be the representation obtained by applying the algorithm to the $\alpha$-block. By the definition of $f$ (and our choice of initial $k$), the algorithm goes through at least $f$ iterations before stopping. In other words, we have $b_j'+1 \leq \frac{l+1}{2}$. Recall that there are two reasons for the algorithm to stop (see Remark \ref{rem_alg} (i)). If none of the factors $\drep{a_{j-1}'}{b_{j-1}'}, \dotsc, \drep{a_1'}{b_1'}$ equals $\drep{1-b_j'}{b_j'}$, then the above representation is isomorphic to
\begin{equation}
\label{eq_6}
\begin{gathered}
\chi_V\Xi \times \chi_V\drep{a_r'}{b_r'} \times \chi_V\drep{a_{r-1}'}{b_{r-1}'} \times \dotsm \times \chi_V\drep{a_j'}{b_j'} \\
\times \chi_V\drep{a_{j-1}'}{b_{j-1}'} \times \dotsm \times \chi_V\drep{a_1'}{b_1'} \times \chi_V\zeta(b_j'+1,k) \rtimes \tau,
\end{gathered}
\end{equation}
by Lemma \ref{lemma_zel0}. If some of the factors are equal to $\drep{1-b_j'}{b_j'}$ (say, those marked with indices $j-h,\dotsc,j-1$), then the representation is isomorphic to
\begin{equation}
\label{eq_7}
\begin{gathered}
\chi_V\Xi \times\chi_V\drep{a_r'}{b_r'} \times \chi_V\drep{a_{r-1}'}{b_{r-1}'} \times \dotsm \times \chi_V\drep{a_j'}{b_j'}\\ \times \chi_V\drep{a_{j-h-1}'}{b_{j-h-1}'} \times \dotsm \times \chi_V\drep{a_{1}'}{b_{1}'} \times \chi_V\zeta(b_j'+1,k) \times \chi_V(\drep{1-b_j'}{b_j'},h)  \rtimes \tau,
\end{gathered}
\end{equation}
where $(\drep{1-b_j'}{b_j'},h)$ denotes the product of $h$ factors $\drep{1-b_j'}{b_j'}$. In either case, $\pi$ is the unique irreducible quotient, as Proposition \ref{prop_unique_q} shows. Therefore, letting $\pi'$ denote the unique irreducible quotient of
\[
\begin{cases}
\chi_V\zeta(b_j'+1,k) \rtimes \tau, &\quad \text{if \eqref{eq_6} holds},\\
\chi_V\zeta(b_j'+1,k) \times \chi_V\drep{1-b_j'}{b_j'} \rtimes \tau &\quad \text{if \eqref{eq_7} holds},
\end{cases}
\]
and setting $\Pi$ to be
\[
\begin{cases}
\drep{a_r'}{b_r'} \times \dotsm \times \drep{a_1'}{b_1'}, &\quad \text{if \eqref{eq_6} holds},\\
\mbox{}\\
\drep{a_r'}{b_r'} \times \dotsm \times \drep{a_j'}{b_j'}\\
\quad \times \drep{a_{j-h-1}'}{b_{j-h-1}'} \times \dotsm \times \drep{a_{1}'}{b_{1}'} \times (\drep{1-b_j'}{b_j'},h-1) &\quad \text{if \eqref{eq_7} holds},
\end{cases}
\]
we get
\begin{equation}
\label{eq_Xi_Pi}
\chi_V\Xi \times \chi_V\Pi \rtimes \pi' \twoheadrightarrow \pi.
\end{equation}
Notice that in case \eqref{eq_7} we had to use Lemma \ref{lem_zel2} to swap the (unique) irreducible quotient of $\zeta(b_j'+1,k) \times \drep{1-b_j'}{b_j'}$ with $(\drep{1-b_j'}{b_j'},h-1)$.

Lemmas \ref{lemma_reverse_temp_thin} (in case \eqref{eq_6}) and \ref{lemma_reverse_temp_thick} (in case \eqref{eq_7}) show that $l(\pi') = l +2f$. This implies that on the going-up tower, $\Theta_{-l-2f}(\pi') = 0$. We now lift $\pi$ to level $-l-2f$ on the going-up tower by repeatedly applying Corollary \ref{cor_theta_epi} to \eqref{eq_Xi_Pi}; we get
\[
\chi_W\Xi \times \chi_W\Pi \rtimes \Theta_{-l-2f}(\pi') \twoheadrightarrow \Theta_{-l-2f}(\pi).
\]
As $\Theta_{-l-2f}(\pi') = 0$, this also implies $\Theta_{-l-2f}(\pi) = 0$ on the going-up tower. The conservation relation now shows that we have $\Theta_{l+2f}(\pi) \neq 0$ on the going-down tower.

To complete the proof, we prove that $\Theta_{l+2f+2}(\pi) = 0$ on the going-down tower. The proof of this fact proceeds by induction on $f$.
We consider the base case, $f=0$. By the definition of of $f$, $f=0$ implies that
\begin{enumerate}[a)]
\item there are no segments ending in $\frac{l+1}{2}$ in the $\alpha$-block (see condition (iii)); or
\item $l>0$, $m_{\phi_\tau}(\chi_VS_l)$ is even, and any segment ending in $\frac{l+1}{2}$ is equal to $[\frac{1-l}{2},\frac{l+1}{2}]$ (cond.(iv)). 
\end{enumerate}
We thus have
\[
\chi_V\Xi \times \chi_V\Pi \rtimes \pi' \twoheadrightarrow \pi
\]
where $\pi' = \tau$ and $\Pi$ is denotes the $\alpha$-block in case (a), whereas $\pi' = L(\chi_VSt_{l+1}\nu^\frac{1}{2}, \dotsc,\allowbreak\chi_VSt_{l+1}\nu^\frac{1}{2}; \tau)$ and $\Pi$ denotes the $\alpha$-block without the factors $\chi_VSt_{l+1}\nu^\frac{1}{2}$, in case (b).

In short, $\Pi$ contains no representations defined by segments ending in $\frac{l+1}{2}$, so we may apply Corollary \ref{cor_theta_epi} to get
\[
\chi_W\Xi \times \chi_W\Pi \rtimes \Theta_{l+2}(\pi') \twoheadrightarrow \Theta_{l+2}(\pi).
\]
We now claim that $\Theta_{l+2}(\pi') = 0$. In case (a) we have $\pi' = \tau$, so this follows from the assumption $l(\tau)=l$. In case (b), this follows from Lemma \ref{lemma_fat_even}. Therefore, $\Theta_{l+2}(\pi') = 0$, implying $\Theta_{l+2}(\pi)=0$. This completes the base case $f=0$.

Now let $f$ be as defined in the statement of the theorem. The last segment in the sequence $[c_1,d_1],\dotsc,[c_f,d_f]$ ends in $\frac{l-1}{2}+f$ (the sequence may not be unique, but the numbers $d_1,\dotsc,d_f$ are, by condition (i)). Since $f$ is defined to be the maximal length, there is no sequence of length $f+1$ satisfying conditions (i)--(iv). Thus, any segment within the $\alpha$-block which ends in $\frac{l-1}{2}+f+1$ must contain all possible segments $[c_f,d_f]$.
This allows us to rearrange the standard module. We introduce some temporary notation:
\begin{align*}
\Pi^+ &= \text{ the part of } \alpha\text{-block containing all segments ending in }\frac{l-1}{2}+f+2\text{ or higher}\\
\Pi^{f+1} &= \text{ the part of } \alpha\text{-block containing all segments ending in }\frac{l-1}{2}+f+1\\
\Pi^f &= \text{ the part of } \alpha\text{-block containing all possible segments }[c_f,d_f]\\
\Pi^- &= \text{ the rest of the } \alpha\text{-block}.
\end{align*}
Thus, initially, we have
\[
\chi_V\Xi \times \chi_V\Pi^+ \times \chi_V\Pi^{f+1} \times  \chi_V\Pi^{f} \times  \chi_V\Pi^{-} \rtimes \tau \twoheadrightarrow \pi.
\]
By the above discussion, $\Pi^{f+1}$ and $\Pi^{f}$ may switch places. We thus have
\[
\chi_V\Xi \times \chi_V\Pi^+ \times \chi_V\Pi^{f} \times  \pi' \twoheadrightarrow \pi
\]
where $\pi'$ denotes the unique irreducible quotient of $\chi_V\Pi^{f+1} \times  \chi_V\Pi^{-} \rtimes \tau$. We may now use Corollary \ref{cor_theta_epi} to show that
\[
\chi_W\Xi \times \chi_W\Pi^+ \times \chi_W\Pi^{f} \rtimes  \Theta_{l+2f+2}(\pi') \twoheadrightarrow \Theta_{l+2f+2}(\pi)
\]
(none of the segments from $\Xi, \Pi^+$ or $\Pi^{f}$ end in $\frac{l-1}{2}+f$). 
We now apply the induction hypothesis to $\pi'$. Notice that the $\alpha$-block of $\pi'$ is equal to $\chi_V\Pi^{f+1} \times  \chi_V\Pi^{-}$. Thus, it may still contain some segments ending in $\frac{l+1}{2}+f$; however, since we have removed all the segments $[c_k,d_k]$ which appear as the top rung of some	 ladder of length $f$ (using the terminology of Remark \ref{rem_alg} (iii)), the longest ladder in the $\alpha$-block of $\pi'$ has length $f-1$. Therefore, by the induction hypothesis, $\Theta_{l+2f}(\pi') = 0$, which further implies $\Theta_{l+2f+2}(\pi') = 0$. Now the above epimorphism shows that $\Theta_{l+2f+2}(\pi)=0$, which we needed to prove.
\end{proof}

\section{The lifts}
\label{sec_lifts}
Having determined the first occurrence index, we now turn to describing the theta lifts explicitly. The following theorem provides a complete description of the non-zero lifts. We continue using the notation of Theorem \ref{theorem_occurrence}; additionally, we let $A$ denote the $\alpha$-block (recall that $\alpha \equiv \frac{l(\tau)-1}{2} \pmod {\mathbb{Z}}$).
\begin{thm}
\label{theorem_appearance}
Let $\pi$ be an irreducible representation of $G_n$ with standard module
\begin{equation}
\label{eq_stdmod}
\chi_V\Xi \times \chi_VA \rtimes \tau
\end{equation}
where $\tau$ is an irreducible tempered representation with parameter $\phi$. Let $l \in \mathbb{Z}$ be such that $\theta_l(\pi)\neq 0$. The standard module of $\theta_l(\pi)$ is obtained by replacing $\chi_V$ with $\chi_W$ and applying a certain transformation to the $\GL$-factors appearing in \eqref{eq_stdmod}. We describe this transformation in terms of ladders. We have three distinct cases:\\
\noindent\textbf{(1) Going-down tower, low rank.}
Let $l \in \mathbb{Z}_{> 0}$ be such that $\theta_l(\pi)\neq 0$ on the going-down tower. Let $[c_t,d_t], [c_{t-1},d_{t-1}], \dotsc, [c_1,d_1]$ be the longest ladder in the $\alpha$-block of $\pi$ such that 
\begin{enumerate}[(i)]
\item $d_t = \frac{l-1}{2}$
\item $d_{i+1} = d_i + 1$ for $i=1,\dotsc,t-1$
\item $\theta_{l-2t}(\tau) \neq 0$ (note that $d_1 = \frac{l+1}{2}-t$; in particular, since $d_1 > 0$, we have $l-2t\geq 0$) 
\end{enumerate}
Notice that we allow $t=0$. The existence of such a ladder is guaranteed by Theorem \ref{theorem_occurrence} and our assumption that $\theta_l(\pi) \neq 0$. The same result shows that when $l-2t > 0$ and $m_\phi(\chi_VS_{l-2t})$ is even, we may choose such a ladder with $c_1 \gneqq 1-d_1$. If there is more than one such ladder, take the one which minimizes the segment widths (in the sense of Remark \ref{rem_alg} (iii)).
\begin{itemize} 
\item Assume $l-2t=0$ or $[c_1,d_1] \neq [1-d_1,d_1]$. Then the standard module of $\theta_l(\pi)$ is obtained from \eqref{eq_stdmod} by replacing $[c_i,d_i]$ with $[c_i,d_i-1]$ for $i=1,\dotsc, t$, and replacing $\tau$ by $\theta_{l-2t}(\tau)$.
\item Assume $l-2t > 0$ and $[c_1,d_1] = [1-d_1,d_1]$. In particular, since $\theta_l(\pi)\neq 0$, Theorem \ref{theorem_occurrence} implies that $m_\phi(\chi_VS_{l-2t})$ is odd. Let $\tau' = L(\chi_VSt_{l-2t+1}\nu^\frac{1}{2};\tau)$. Then $\theta_{l-2t+2}(\tau')$ is non-zero and tempered.

The standard module of $\theta_l(\pi)$ is obtained from \eqref{eq_stdmod} by removing $[c_1,d_1]$, replacing $[c_i,d_i]$ with $[c_i,d_i-1]$ for $i=2,\dotsc, t$ and replacing $\tau$ by $\theta_{l-2t+2}(\tau')$.
\end{itemize}
In the special going-down case when $l=0$, there is no ladder; we only replace $\tau$ by $\theta_0(\tau)$.
After making these changes, we sort the $\alpha$-block if needed.

\noindent\textbf{(2) Going-down tower, high rank.} Let $l > 0$ and consider $\theta_{-l}(\pi)$ on the going-down tower. Recall that $\kappa \in \{1,2\} \equiv l$. The standard module of $\theta_{-l}(\pi)$ is obtained from \eqref{eq_stdmod} by replacing $\tau$ with $\theta_{\kappa-2}(\tau)$ and inserting $\chi_W|\cdot|^\frac{3-\kappa}{2}, \chi_W|\cdot|^\frac{5-\kappa}{2}, \dotsc, \chi_W|\cdot|^\frac{l-1}{2}$ among the $\GL$-factors so that the $\alpha$-block remains sorted.

\noindent\textbf{(3) Going-up tower.} Let $l > 0$ such that $\theta_{-l}(\pi)\neq 0$ on the going-up tower and let $l_0 = l(\tau)$. The conservation relation implies that $\theta_{-l_0-2}$ is the first lift of $\tau$ on the going-up tower. We may assume that $l_0 \geq 0$; the only remaining possibility is $l_0 = -1$, but then $\theta_{-1}(\pi)$ is the first lift on both towers, and the lifts on both towers are treated by the previous case (going-down tower, high rank).

Let $t=\frac{l-l_0}{2}-1$ and let $[c_1',d_1'], [c_2',d_2'], \dotsc, [c_t',d_t']$ be the ladder in the $\alpha$-block which maximizes the segment widths such that $d_1' = \frac{l_0+1}{2}$; if there is no such ladder of length $t$, we take the longest available ladder and use empty segments to achieve length $t$ (see (iv) of Remark \ref{rem_alg}). The standard module of $\theta_{-l}(\pi)$ is obtained from \eqref{eq_stdmod} by replacing $[c_i',d_i']$ with $[c_i',d_i'+1]$ for $i=1,\dotsc,t$ so that the $\alpha$-block remains sorted. Additionally, 
\begin{itemize}
\item assume that $l_0 = 0$ or that $m_\phi(\chi_VS_{l_0})$ is odd. In that case $\tau$ is replaced by $\theta_{-2-l_0}(\tau)$, which is tempered.
\item Assume that $l_0 > 0$ and that $m_\phi(\chi_VS_{l_0})$ is even. In that case $\theta_{-2-l_0}(\tau)$ is equal to $L(\chi_W\text{St}_{l_0+1}\nu^\frac{1}{2}; \sigma)$ for a certain tempered representation $\sigma$. We insert $\chi_W\text{St}_{l_0+1}\nu^\frac{1}{2}$ among the $\GL$-factors so that the $\alpha$-block remains sorted, and replace $\tau$ by $\sigma$.
\end{itemize}
\end{thm}

The proof of (2) and (3) will follow from (1). The proof of (1) contains many details which handle exceptional cases and deal with various situations arising from possibly complicated structure of the standard module of $\pi$. In order to present the main idea of proof without obscuring it with technical details, we now explain it with a simple example.
\begin{exmp}
Let $\tau \in \Irr(\text{Sp}(W_n))$ be a discrete series representation with $l(\tau) = 3$. Let $\pi$ be the unique irreducible quotient of
\[
\chi_V\drep{4}{5} \times \chi_V\drep{3}{4} \times \chi_V\drep{2}{3}   \times \chi_V\drep{1}{2} \rtimes \tau.
\]
By Theorem \ref{theorem_occurrence} we have $l(\pi) = 11$, and Theorem \ref{theorem_appearance} predicts that $\theta_{11}(\pi)$ is the unique irreducible quotient of
\begin{equation}
\label{eq_def}
|\cdot|^4 \times |\cdot|^3 \times |\cdot|^2  \times |\cdot|^1 \rtimes \theta_3(\tau).
\end{equation}
Notice that the length of the ladder described in Theorem \ref{theorem_appearance} (1) for $l=11$ is $t = 4$. We let $\pi'$ denote the irreducible quotient of \eqref{eq_def}; we wish to show that $\theta_{-11}(\pi') = \pi$. Letting $\sigma$ denote $\theta_{3}(\tau)$ and applying Corollary \ref{cor_theta_epi} to
\[
|\cdot|^4 \times |\cdot|^3 \times |\cdot|^2  \times |\cdot|^1 \rtimes \sigma \twoheadrightarrow \pi',
\]
we get
\begin{equation}
\label{eq_ex_2}
\chi_V|\cdot|^4 \times \chi_V|\cdot|^3 \times \chi_V|\cdot|^2  \times \chi_V|\cdot|^1 \rtimes \Theta_{-11}(\sigma) \twoheadrightarrow \theta_{-11}(\pi').
\end{equation}
For tempered representations, the complete description of subquotients of full theta lifts is not known. This causes major technical complications in the proof, and is the reason why we switch from \eqref{eq_sigma_back} to \eqref{eq_stdmod_lift} in the proof below. However, in this case, $\sigma$ is a discrete series representation, so we may use the results of \cite{Muic_theta_discrete_Israel} to describe all possible subquotients of $\Theta_{-11}(\sigma)$. Thus, any subquotient of $\Theta_{-11}(\sigma)$ is the irreducible quotient of $\zrep{a}{5} \rtimes \tau'$, where $\tau'$ is a tempered subquotient of $\Theta_{1-2a}(\sigma)$, $a \in \{2,3,4,5,6\}$. In particular, $\theta_{-11}(\sigma)$ is obtained when $a = 2$ (in that case $\Theta_{-3}(\sigma)$ is irreducible, i.e. $\Theta_{-3}(\sigma) = \theta_{-3}(\sigma) = \tau)$.

A critical step in our proof is proving that $\theta_{-11}(\sigma)$ is the irreducible subquotient of $\Theta_{-11}(\sigma)$ which participates in the above epimorphism \eqref{eq_ex_2}---see Lemma \ref{lemma_small_theta}. To prove this, we assume the contrary, i.e.~that
\[
\chi_V|\cdot|^4 \times \chi_V|\cdot|^3 \times \chi_V|\cdot|^2  \times \chi_V|\cdot|^1 \times \chi_V\zrep{a}{5} \rtimes \tau' \twoheadrightarrow \theta_{-11}(\pi')
\]
for some $a > 2$. For example, if $a = 4$, this means
\[
\chi_V|\cdot|^4 \times \chi_V|\cdot|^3 \times \chi_V|\cdot|^2  \times \chi_V|\cdot|^1 \times \chi_V\zrep{4}{5} \rtimes \tau' \twoheadrightarrow \theta_{-11}(\pi')
\]
But we may then show (using Proposition \ref{prop_unique_q}) that the left-hand side above has a unique irreducible quotient, and that
\[
\chi_V\drep{4}{5} \times \chi_V\drep{3}{4} \times \chi_V|\cdot|^2  \times \chi_V|\cdot|^1 \rtimes \tau' \twoheadrightarrow \theta_{-11}(\pi').
\]
We now notice that the ladder described in Theorem \ref{theorem_appearance} (1) (for $l= 11$) in the $\alpha$-block of the above representation is of length $t = 2$. The fact that this ladder is shorter than the original ladder in $\pi$ enables an inductive proof: if we assume that the Theorem \ref{theorem_appearance} (1) holds in any situation where the ladder is of length $t < 4$, then we may apply it to the above representation. We get that
$\pi' = \theta_{11}(\theta_{-11}(\pi'))$ is the unique irreducible quotient of
\[
|\cdot|^4 \times |\cdot|^3 \times |\cdot|^2  \times |\cdot|^1 \rtimes \theta_7(\tau').
\]
However, comparing this representation with \eqref{eq_def}, we easily prove that $\theta_7(\tau')$ cannot be equal to $\theta_3(\tau)$. We have thus arrived at a contradiction, proving that $a=2$, i.e. that the subquotient which participates in \eqref{eq_ex_2} is equal to $\theta_{-11}(\sigma)$.

This explains the main idea in the inductive step which allows us to expand the statement of Theorem \ref{theorem_appearance} (1) from representations with ladder length $t < 4$ to those with ladder length $t = 4$. The rest of the proof is more complicated when $\tau$ is tempered (but not in discrete series); however, the idea outlined above is the main ingredient in the proof even in case when $\tau$ is not in discrete series.
\end{exmp}

\begin{proof}[Proof of Theorem \ref{theorem_appearance}]
We consider the three cases separately.

\noindent\textbf{(1) Going-down tower, low rank.} The remaining two cases will follow from this one. We fix $l \geq 0$. We induce on the length of the ladder $t$ from the statement of Theorem \ref{theorem_appearance} and (if $d_1 > \frac{1}{2}$) on the number of segments in the $\alpha$-block which end in $d_1$, but are not equal to $[1-d_1,d_1]$ (we remind the reader that this $d_1$ depends on both $l$ and the length of the maximal ladder, i.e.~$t$).
In the base case, we have $t=0$, i.e.~there is no ladder. Thus, either $l=0$, or there is no segment ending in $\frac{l-1}{2}$. In any case, $\theta_l(\tau) \neq 0$ by \textbf{(1)}-(iii). We may therefore set $\pi'$ to be the unique irreducible quotient of $\chi_W\Xi \times \chi_WA \rtimes \theta_l(\tau)$:
\[
\chi_W\Xi \times \chi_WA \rtimes \theta_l(\tau) \twoheadrightarrow \pi'.
\]
According to Theorem \ref{theorem_appearance}, $\pi'$ should be equal to $\theta_l(\pi)$. Indeed, this follows easily: repeatedly applying Corollary \ref{cor_theta_epi} to the above epimorphism, we get
\[
\chi_V\Xi \times \chi_VA \rtimes \Theta_{-l}(\theta_l(\tau)) \twoheadrightarrow \theta_{-l}(\pi').
\]
Since $\Theta_{-l}(\theta_l(\tau))$ is in fact irreducible, we have $\Theta_{-l}(\theta_l(\tau)) = \tau$. Thus, comparing the left-hand side of the above epimorphism with \eqref{eq_stdmod}, we see that $\theta_{-l}(\pi') = \pi$.
Therefore, $\theta_l(\pi) = \pi'$, as desired. This completes the base case.

Now assume that the description of the low-rank lifts is true whenever the length of the ladder in Theorem \ref{theorem_appearance} (1) is strictly less than $t$. We prove that the description of the lifts also holds when the ladder is of length $t$. Thus, let $\pi$ be as in the statement of the theorem.

First, we assume that $l-2t > 0$ and $[c_1,d_1] = [1-d_1,d_1]$. According to the statement of the theorem, we need to show that $\theta_l(\pi)$ is the unique irreducible quotient of
\begin{equation}
\label{eq_stdmod2}
\chi_W\Xi \times \chi_WA' \rtimes \theta_{l-2t+2}(\tau'),
\end{equation}
where $\tau' = L(\chi_VSt_{l-2t+1}\nu^\frac{1}{2};\tau) = L(\chi_V\delta(1-d_1,d_1);\tau)$. Here $A'$ denotes the $\alpha$-block obtained from $A$ by applying the transformation described in Theorem \ref{theorem_appearance}; namely, $[c_i,d_i]$ is replaced by $[c_i,d_i-1]$ for $i=2,\dotsc,t$, and $[c_1,d_1]$ is omitted (it is now a part of $\tau'$). We let $\pi'$ denote the unique irreducible quotient of the standard representation in \eqref{eq_stdmod2}. Also, to simplify notation, set $\sigma = \theta_{l-2t+2}(\tau')$. We thus have
\begin{equation}
\label{eq_sigma}
\chi_W\Xi \times \chi_WA' \rtimes \sigma \twoheadrightarrow \pi'
\end{equation}
To prove that \eqref{eq_stdmod2} is indeed the standard module of $\theta_l(\pi)$, we show that $\theta_{-l}(\pi') = \pi$. Applying Corollary \ref{cor_theta_epi} to the above epimorphism, we lift $\pi'$ back to the tower $(W_n)$:
\begin{equation}
\label{eq_sigma_back}
\chi_V\Xi \times \chi_VA' \rtimes \Theta_{-l}(\sigma) \twoheadrightarrow \theta_{-l}(\pi').
\end{equation}
We would like to show that the subquotient of $ \Theta_{-l}(\sigma)$ which participates in the above epimorphism is $\theta_{-l}(\sigma)$. However, as subquotients of full theta lifts of tempered representations are not fully understood, we prefer to work with discrete series representations. Recall that $\sigma = \theta_{l-2t+2}(\tau')$ is indeed a non-zero tempered representation (cf.~the proof of Lemma \ref{lemma_reverse_temp_thick} or Theorem 4.5 of \cite{Atobe_Gan}). Therefore, there are discrete series representations $\delta_1,\dotsc, \delta_d, \sigma'$ such that $\chi_W\delta_1 \times \dotsm \times \chi_W\delta_d \rtimes \sigma' \twoheadrightarrow \sigma$; note that at least one of the $\delta_i$ equals $\delta(1-d_1,d_1-1)$. Setting $\Delta = \delta_1 \times \dotsm \times \delta_d$ we may thus replace $\sigma$ with $\chi_W\Delta \rtimes \sigma'$ in \eqref{eq_sigma}. Using Corollary \ref{cor_theta_epi} again, we arrive at
\begin{equation}
\label{eq_stdmod_lift}
\begin{aligned}
\chi_V\Xi \times \chi_VA' \times \chi_V\Delta\rtimes \Theta_{-l}(\sigma') \twoheadrightarrow \theta_{-l}(\pi').
\end{aligned}
\end{equation}
instead of \eqref{eq_sigma_back}. A key observation in our proof is the following:
\begin{lem}
\label{lemma_small_theta}
The subquotient $\Theta_{-l}(\sigma')$ which participates in  \eqref{eq_stdmod_lift} is none other than $\theta_{-l}(\sigma')$.
\end{lem}

\noindent The proof of this lemma is the part of the induction step which relies on the induction hypothesis. However, since it also references the ensuing part of the proof of Theorem \ref{theorem_appearance} we postpone it until after the proof of this theorem to improve readability (see Appendix).

The above lemma shows that we may replace $\Theta_{-l}(\sigma')$ with $\theta_{-l}(\sigma')$ in \eqref{eq_stdmod_lift}, thus getting
\[
\chi_V\Xi \times \chi_VA' \times \chi_V\Delta\times \chi_V\zrep{\frac{l+1}{2}-t}{\frac{l-1}{2}} \rtimes \theta_{2t-l}(\sigma') \twoheadrightarrow \theta_{-l}(\pi').
\]
Since at least one of the $\delta_i$'s from the definition of $\Delta$ is equal to $\text{St}_{l-2t}$, this leads to two possibilities. One possibility is that $\zrep{\frac{l+1}{2}-t}{\frac{l-1}{2}}$ may switch places with $\Delta$, so that
\begin{align*}
\chi_V\Xi \times \chi_VA' \times \chi_V\zrep{\frac{l+1}{2}-t}{\frac{l-1}{2}} \times \chi_V\Delta \rtimes \theta_{2t-l}(\sigma') \twoheadrightarrow \theta_{-l}(\pi'),
\end{align*}
but we also have to take into consideration the other option, i.e.
\begin{align*}
\chi_V\Xi \times \chi_VA' \times \chi_VL \times \chi_V\Delta' \rtimes \theta_{2t-l}(\sigma') \twoheadrightarrow \theta_{-l}(\pi'),
\end{align*}
where $\Delta = \text{St}_{l-2t} \times \Delta'$, and $L$ is the unique irreducible quotient of $\zrep{\frac{l+3}{2}-t}{\frac{l-1}{2}}\times \text{St}_{l-2t+1}\nu^\frac{1}{2}$. Letting $\tau_1'$ (resp.~$\tau_2'$) be the appropriate (necessarily tempered) irreducible subquotient of $\chi_V\Delta \rtimes \theta_{2t-l}(\sigma')$ (resp.~$\chi_V\Delta' \rtimes \theta_{2t-l}(\sigma')$), the above discussion shows that $\theta_{-l}(\pi')$ is an irreducible quotient of
\begin{equation}
\label{eq_thin}
\chi_V\Xi \times \chi_VA' \times \chi_V\zrep{\frac{l+1}{2}-t}{\frac{l-1}{2}} \rtimes \tau_1'
\end{equation}
or
\begin{equation}
\label{eq_thick}
\chi_V\Xi \times \chi_VA' \times \chi_VL \rtimes \tau_2'.
\end{equation}
We now claim the following:
\begin{lem}
\label{lemma_unique_quotient_return}
Both \eqref{eq_thin} and \eqref{eq_thick} possess a unique irreducible quotient. 
\end{lem}
\begin{proof}
We first address \eqref{eq_thin}. We use Remark \ref{rem_alg} (v) to show that \eqref{eq_thin} is the output of the algorithm of Section \ref{subs_algorithm}; if that is the case, Proposition \ref{prop_unique_q} shows that the irreducible quotient is unique.

First, let $[c_1',d_1']$ be the longest segment in $A'$ which ends in $d_1' = \frac{l-1}{2}-t$. Any segments appearing in $A'$ to the right of $[c_1',d_1']$ are not linked to $[\frac{l+1}{2}-t, \frac{l-1}{2}]$, so we can move $\zrep{\frac{l+1}{2}-t}{\frac{l-1}{2}}$ to the immediate right of $[c_1',d_1']$. By Remark (iv), we consider the ladder $[c_1',d_1'], \dotsc, [c_t',d_t']$ which maximizes segment lengths. Note that $d_1' = \frac{l-1}{2}-t, \dotsc, d_t' = \frac
{l-3}{2}$, i.e.~$d_i' = d_i-1$ for all $i$. We need to show that $[c_t',\frac{l-1}{2}]$ is the shortest segment ending in $\frac{l-1}{2}$ in the $\alpha$-block after we apply the transformation described in Remark \ref{rem_alg} (iv) to this ladder. (Note that it need not be the unique shortest segment, i.e.~the inequality is not necessarily strict)

In order to prove that $[c_t',d_t'+1]$ is the shortest, it suffices to prove that it is shorter than or equal to $[c_t,d_t]$. Indeed, $[c_t,d_t]$ was the shortest segment in $A$ with $d_t = \frac{l-1}{2}$. Therefore, if $[c_t',d_t'+1]$ is shorter than $[c_t,d_t]$, it is also shorter than any other segment ending in $\frac{l-1}{2}$.

To prove the desired inequality, we prove by induction that $[c_i',d_i'+1]$ is shorter than (or equal to) $[c_i,d_i]$ for all $i = 1, \dotsc, t$. For the base case $i=1$, recall that $[c_1,d_1] = [1-d_1,d_1]=[-d_1',d_1'+1]$. The segment $[c_1',d_1']$ is obviously strictly shorter than $[-d_1',d_1']$, so $[c_1',d_i'+1]$ is shorter than $[-d_1',d_1'+1] = [c_1,d_1]$.


Now assume that $[c_{i-1}',d_{i-1}'+1]$ is shorter than $[c_{i-1},d_{i-1}]$ for some $i > 2$. By Remark \ref{rem_alg} (iv), $[c_i',d_i']$ is chosen to be the longest segment with $d_i' = d_{i-1}'+1$ which is linked to $[c_{i-1}',d_{i-1}']$. We want to show that $c_i'\ge c_i.$ If $[c_i',d_i']$ is obtained by shortening the segment $[c_i,d_i]$ we have $c_i'=c_i$ and we are finished. If not, it means that $[c_i',d_i']$ is just one of the segements which end in $d_i-1.$ If this is one of the segments which switch places with $[c_i,d_i]$ in the algorithm, we are also finished. Otherwise, since $[c_{i-1},d_i-1]$ is the first one which does not switch places with $[c_i,d_i],$ we must have $c_i'\le c_{i-1}.$ But, by the induction hypothesis, we have $c_i'\le c_{i-1}\le c_{i-1}',$ so that $c_i'\le c_{i-1}';$ a contradiction.

The proof in case \eqref{eq_thick} is much shorter: we only need to notice that \eqref{eq_thick} closely resembles the representation obtained by applying the algorithm to $\pi$. Indeed, notice that $L$ can switch places with any other factors of the form $\text{St}_{l-2t+1}\nu^\frac{1}{2}$ which appear in $A'$ (see Lemma \ref{lem_zel2}). We may thus move $L$ to the left of all such factors and write $\zrep{\frac{l+3}{2}-t}{\frac{l-1}{2}}\times \text{St}_{l-2t+1}\nu^\frac{1}{2}$ instead of $L$. On the one hand, the representation obtained this way obviously possesses \eqref{eq_thick} as a quotient. On the other hand, this representation we have just obtained has the exact same $\alpha$-block as the one obtained by applying the algorithm to $\pi$. Therefore, Proposition \ref{prop_unique_q} guarantees that it possesses a unique irreducible quotient. Consequently, \eqref{eq_thick} also has a unique irreducible quotient.
\end{proof}
We now prove that \eqref{eq_thin} is impossible, whereas \eqref{eq_thick} leads to the desired conclusion that $\theta_{-l}(\pi') = \pi$. If $\theta_{-l}(\pi')$ is a quotient of \eqref{eq_thin}, we now know that it is the unique quotient; therefore, the subquotient of $\chi_V\zrep{\frac{l+1}{2}-t}{\frac{l-1}{2}} \rtimes \tau_1'$ which participates in the above epimorphism is its unique irreducible quotient. Similarly, if $\theta_{-l}(\pi')$ is a quotient of \eqref{eq_thick}, the subquotient of $\chi_VL \rtimes \tau_2'$ which participates in the epimorphism is its unique irreducible quotient. Also note that these irreducible subquotients appear in $\Theta_{-l}(\sigma)$. We thus arrive at the following conclusions:
\begin{itemize}
\item if $\theta_{-l}(\pi')$ is a quotient of \eqref{eq_thin}, then the subquotient of $\Theta_{-l}(\sigma)$ which participates in \eqref{eq_sigma_back} is of the form
\begin{equation}
\label{eq_thin2}
L(\chi_V|\cdot|^{\frac{l-1}{2}},\dotsc,\chi_V|\cdot|^{\frac{l+1}{2}-t}; \tau_1')
\end{equation}
\item if $\theta_{-l}(\pi')$ is a quotient of \eqref{eq_thick}, then the subquotient of $\Theta_{-l}(\sigma)$ which participates in \eqref{eq_sigma_back} is of the form
\begin{equation}
\label{eq_thick2}
L(\chi_V|\cdot|^{\frac{l-1}{2}},\dotsc,\chi_V|\cdot|^{\frac{l+3}{2}-t},\chi_V\text{St}_{l-2t+1}\nu^\frac{1}{2};\tau_2')
\end{equation}
\end{itemize}
In both cases, we make use of the following lemma.
\begin{lem}
\label{lemma_reduce}
Let $\sigma \in \Irr(H_m)$ be a tempered representation. Let $k>0$ and let $\xi$ be an irreducible subquotient of $\Theta_{-k}(\sigma)$, and at the same time a quotient of $\chi_V|\cdot|^{\frac{k-1}{2}} \rtimes \xi'$ for some irreducible representation $\xi'$. Then $\xi'$ is a subquotient of $\Theta_{2-k}(\sigma)$; in particular, $\Theta_{2-k}(\sigma)\neq 0$.
\end{lem}
\begin{proof}
This lemma requires only a slight modification of the arguments used in the proof of Theorem 4.1, \cite{Muic_theta_discrete_Israel}; we leave this to the reader.
\end{proof}
Now if \eqref{eq_thin2} holds, we may use the above lemma repeatedly (first with $k=l$, then $k=l-2$, etc.) to show that $\Theta_{2t-l}(\sigma) \neq 0$.
However, recalling that $\sigma = \theta_{l-2t+2}(\tau')$, we easily see that $\theta_{2t-l-2}(\sigma)$ is the first non-zero lift on the going up tower; therefore, $\Theta_{2t-l}(\sigma) \neq 0$ is impossible. We have thus ruled out the possibility that \eqref{eq_thin2} holds.

Using the same argument, \eqref{eq_thick2} implies that $L(\chi_V\text{St}_{l-2t+1}\nu^\frac{1}{2};\tau_2')$ is a subquotient of $\Theta_{2t-l-2}(\sigma)$. We now show the following:
\begin{lem}
\label{lemma_unique_subquotient}
$\Theta_{2t-l-2}(\sigma)$ possesses only one subquotient of the form $L(\chi_V\text{St}_{l-2t+1}\nu^\frac{1}{2};\tau_2')$, and that it is equal to $\tau'$, i.e.~$L(\chi_V\text{St}_{l-2t+1}\nu^\frac{1}{2};\tau)$.
\end{lem}
The technical proof of this lemma does not contain details relevant to the proof of the theorem, so we again postpone it to improve readability (see Appendix).

Let us review our results so far: we have shown that $\theta_{-l}(\pi')$ cannot be a quotient of \eqref{eq_thin2}; thus, it is a quotient of \eqref{eq_thick2}. Furthermore, the above lemma shows that in \eqref{eq_thick2} we have $\tau_2' = \tau$. Repeating the arguments from the second part of the proof of Lemma \ref{lemma_unique_quotient_return} (where we proved that \eqref{eq_thick} possesses a unique irreducible quotient) we now show that $\theta_{-l}(\pi')$ has the same standard module as $\pi$---this time we know not only that the $\alpha$-block is the same, but also that $\tau_2' = \tau$. Therefore, $\theta_{-l}(\pi') = \pi$, as desired.

This completes the inductive step and proves Theorem \ref{theorem_appearance} (1) in case when $l-2t > 0$ and $[c_1,d_1] = [1-d_1,d_1]$. The inductive step in case $l-2t = 0$ or $c_1 > 1-d_1$ follows along the same lines, so we only stress the main points where it diverges from the previous case.

Assume that $\pi$ is an irreducible representation as in the statement od Theorem \ref{theorem_appearance}. We now assume that $l-2t=0$ or that $c_1 > 1-d_1$ and describe the inductive step.  According to the statement of the theorem, we need to prove that $\theta_{l}(\pi)$ is the unique irreducible quotient of
\begin{equation}
\label{eq_stdmod3}
\chi_W\Xi \times \chi_WA' \rtimes \theta_{l-2t}(\tau);
\end{equation}
in this case, $A'$ is obtained from $A$ by replacing $[c_i,d_i]$ with $[c_i,d_i-1]$ for all $i = 1,\dotsc, t$. Letting $\pi'$ denote the unique irreducible quotient of the above representation, we wish to prove that $\theta_{-l}(\pi') = \pi$. Setting $\sigma = \theta_{l-2t}(\tau)$ and $\sigma = \chi_V\Delta \rtimes \sigma'$ as in the previous case, we get an analogue of \eqref{eq_stdmod_lift}:
\[
\chi_V\Xi \times \chi_VA' \times \chi_V\Delta\rtimes \Theta_{-l}(\sigma') \twoheadrightarrow \theta_{-l}(\pi').
\]
Now the proof of Lemma \ref{lemma_small_theta} works in this case as well, so we get
\[
\chi_V\Xi \times \chi_VA' \times \chi_V\Delta\rtimes \theta_{-l}(\sigma') \twoheadrightarrow \theta_{-l}(\pi').
\]
Recall that $\theta_{-l}(\sigma')$ is the unique irreducible quotient of $\chi_V\zrep{\frac{l+1}{2}-t}{\frac{l-1}{2}} \rtimes \theta_{2t-l}(\sigma')$. Since one of the $\delta_i$'s from which $\Delta$ is induced may be equal to $\text{St}_{l-2t}$ (in the previous case this was necessary; here, it is only a possibility), we do not know if $\zrep{\frac{l+1}{2}-t}{\frac{l-1}{2}}$ can freely switch places with $\Delta$. We thus have two options, analogous to \eqref{eq_thin} and \eqref{eq_thick}:
\begin{equation}
\label{eq_thin3}
\chi_V\Xi \times \chi_VA' \times \chi_V\zrep{\frac{l+1}{2}-t}{\frac{l-1}{2}} \rtimes \tau_1 \twoheadrightarrow \theta_{-l}(\pi')
\end{equation}
or
\begin{equation}
\label{eq_thick3}
\chi_V\Xi \times \chi_VA' \times \chi_VL \rtimes \tau_2 \twoheadrightarrow \theta_{-l}(\pi').
\end{equation}
Here $\tau_1$ (resp.~$\tau_2$) is an irreducible tempered subquotient of $\chi_V\Delta \rtimes \theta_{2t-l}(\sigma')$ (resp.~$\chi_V\Delta' \rtimes \theta_{2t-l}(\sigma')$), just like in the previous case.

In this case, we prove that \eqref{eq_thick3} is impossible, whereas \eqref{eq_thin3} holds.

First, observe that if \eqref{eq_thin3} is true, then $\theta_{-l}(\pi') = \pi$, as desired. Indeed, the left-hand side \eqref{eq_thin3} has a unique irreducible quotient; this is shown using the same argument which was used to prove the analogous claim for \eqref{eq_thick} in the second part of the proof of Lemma \ref{lemma_unique_quotient_return}. Then, Lemma \ref{lemma_reduce} may be used to show that in \eqref{eq_thin3} we have $\tau_1 = \tau$. Thus, $\theta_{-l}(\pi')$ has not only the same $\alpha$-block, but also the same tempered part as $\pi$. Consequently, $\theta_{-l}(\pi') = \pi$, which we wanted to prove.

It thus remains to prove that \eqref{eq_thick3} cannot hold. We prove this fact by induction on the number of segments in $A$ which end in $d_1$, but are different (i.e.~strictly shorter) than $[1-d_1,d_1]$. For the base case, assume that there is only one such segment. The arguments we use here are similar to those we use in the proof of Lemma \ref{lemma_small_theta}: using Lemma \ref{lem_len_k}, we show that the ladder which appears in the $\alpha$-block of $\theta_{-l}(\pi')$ is not longer than $t$. However, even if the length is $t$, we may argue as follows: since $A$ contains only one segment ending in $d_1$ and different from $[1-d_1,d_1]$, \eqref{eq_thick3} easily implies that there are no such segments in the $\alpha$-block of $\theta_{-l}(\pi')$. Therefore, we may apply Theorem \ref{theorem_appearance} to $\theta_{-l}(\pi')$---either because the length of the ladder is less than $t$, or by the previous case, because the lowest rung of the ladder is now $[1-d_1,d_1]$.

However, the description of the standard module of $\theta_l(\theta_{-l}(\pi'))$ from Theorem \ref{theorem_appearance} does not match with the one provided in \eqref{eq_stdmod3} (one checks this by comparing tempered parts of the standard modules, just like we did in Lemma \ref{lemma_small_theta}). This proves that \eqref{eq_thick3} is impossible in the base case. Therefore, \eqref{eq_thin3} holds and we may apply the above discussion to infer $\theta_{-l}(\pi') = \pi$.

The inductive step is the same as the base case. For the inductive hypothesis, assume that Theorem \ref{theorem_appearance} (1) holds whenever $A$ contains less than $m$ segments which end in $d_1$ and are shorter than $[1-d_1,d_1]$. Now if $\pi$ has $m$ such segments in $A$, we use the same arguments as above to show that the $\alpha$-block of $\theta_{-l}(\pi')$ has exactly $m-1$ such segments if \eqref{eq_thick3} holds. We may therefore use the inductive hypothesis and apply the description from Theorem \ref{theorem_appearance} to compute $\theta_l(\theta_{-l}(\pi'))$. Comparing this description to the one we have in \eqref{eq_stdmod3}, we see that this is impossible---for example, if there are $m'$ segments of the form $[1-d_1,d_1]$ in $A$, then there are $m'$ of them in \eqref{eq_stdmod3}, but $m'+1$ in the $\alpha$-block of $\theta_l(\theta_{-l}(\pi'))$ as described by Theorem \ref{theorem_appearance}. Thus, we arrive at a contradiction, and infer that \eqref{eq_thick3} cannot hold. Again, this means that \eqref{eq_thin3} holds, and we may finish the proof by repeating the argument described above. This completes the inductive step in case $[c_1,d_1] \neq [1-d_1,d_1]$.

We have therefore completed the inductive step, thereby showing that Theorem \ref{theorem_appearance} (1) holds when the length of the ladder equals $t$. Thus, Theorem \ref{theorem_appearance} (1) is proven.

We now prove parts (2) and (3), both of which follow directly from (1).

\bigskip

\noindent\textbf{(2) Going-down tower, high rank.} Let $\pi'$ be the quotient of the standard representation described in Theorem \ref{theorem_appearance} (2). We need to show that $\theta_{-l}(\pi) = \pi'$. To this end, we may now use Theorem \ref{theorem_appearance} (1) to compute $\theta_l(\pi')$; it is obvious that in this case, the relevant ladder in the $\alpha$-block of $\pi'$ is $|\cdot|^{\frac{l-1}{2}}, |\cdot|^\frac{l-3}{2}, \dotsc, |\cdot|^{\frac{3-\kappa}{2}}$---precisely the ladder we had added to the $\alpha$-block of $\pi$ to obtain $\pi'$. According to Theorem \ref{theorem_appearance} (1), we should remove this ladder, and replace $\theta_{\kappa-2}(\tau)$ with $\theta_{2-\kappa}(\theta_{\kappa-2}(\tau)) = \tau$. Therefore, the standard module of $\theta_l(\pi')$ is equal to the standard module of $\pi$, which implies $\theta_l(\pi') = \pi$, i.e.~$\theta_{-l}(\pi) = \pi'$, as desired.

\bigskip

\noindent\textbf{(3) Going-up tower} The proof in this case repeats the steps of the last one; however, it is slightly less obvious, so we elaborate. Again, we let $\pi'$ be the quotient of the standard representation described in Theorem \ref{theorem_appearance} (3); we prove that $\theta_{-l}(\pi) = \pi'$. Once more, we do this by using Theorem \ref{theorem_appearance} (1) to compute $\theta_l(\pi')$. In the previous case, it was obvious that the relevant ladder in the $\alpha$-block of $\pi'$ is the one we have just created. Here, this is not as evident, but is still valid. Once we prove this fact, the claim $\theta_l(\pi') = \pi$ will follow, just like in (2). Note that part (3) of the Theorem has two cases, depending on $m_\phi(\chi_VS_{l_0})$. We prove the first one; the second one is similar and we leave it to the reader.

Thus, assume that $l_0 = 0$ or that $m_\phi(\chi_VS_{l_0})$ is odd. We let $[c_1',d_1'], [c_2',d_2'], \dotsc, [c_t',d_t']$ be the ladder as described in Theorem \ref{theorem_appearance} (3); we set $[c_i,d_i] = [c_i',d_i'+1]$. As explained above, in order to prove that $\theta_l(\pi') = \pi$, it remains to show that, in the $\alpha$-block of $\pi'$, $[c_t,d_t], [c_{t-1},d_{t-1}], \dotsc, [c_1,d_1]$ is the ladder of maximum length, and of minimal width among those of maximum length, such that $d_t = \frac{l-1}{2}$. We thus need to prove two things: first, that $[c_t,d_t]$ is the shortest segment ending in $\frac{l-1}{2}$ in the $\alpha$-block of $\pi'$; and secondly, that there is no segment $[c_{0},d_{0}]$ linked to $[c_1,d_1]$ such that $d_0 = d_1-1$.

First, if there were a segment $[c,d]$ with $d=\frac{l-1}{2}$ strictly shorter than $[c_t,d_t]$ in the $\alpha$-block of $\pi'$, that segment would also have to appear in the $\alpha$-block of $\pi$. Then $[c,d]$, $[c_t',d_t']$, $[c_{t-1}',d_{t-1}'],\allowbreak \dotsc,\allowbreak [c_1',d_1']$ would constitute a ladder inside the $\alpha$-block of $\pi$ satisfying the conditions of Theorem \ref{theorem_occurrence}. Theorem \ref{theorem_occurrence} would thus imply that $l(\pi) \geq l$, which would in turn imply $\theta_{-l}(\pi) = 0$ on the going-up tower, by the conservation relation. This contradicts our assumption that $\theta_{-l}(\pi)\neq 0$. Therefore, no such segment $[c,d]$ exists, i.e.~$[c_t,d_t]$ really is the shortest segment ending in $\frac{l-1}{2}$. We may now deduce that $[c_t,d_t], \dotsc, [c_1,d_1]$ is the (initial part of the) ladder in the $\alpha$-block of $\pi'$ to which Theorem \ref{theorem_appearance} (1) is applied (see Remark \ref{rem_alg} (v)). It remains to see that this is also the whole ladder, i.e. that there is no segment $[c_0,d_0]$ as described above. This follows directly from our construction: if there were such a segment, then the fact that $[c_0,d_0]$ is linked to $[c_1,d_1]$ would imply that $[c_0,d_0]$ is longer than $[c_1',d_1']$ (note that $d_1' = d_0$). This contradicts our choice of $[c_1',d_1']$ as the longest segment ending in $\frac{l_0+1}{2}$.

This completes the proof of the third and last part of Theorem \ref{theorem_appearance}.
\end{proof}

\section*{Appendix: proofs of Lemmas \ref{lemma_small_theta} and \ref{lemma_unique_subquotient}}
\begin{proof}[Proof of Lemma \ref{lemma_small_theta}]
Before proving the lemma, we need to establish a few facts about $\sigma'$. First, by definition, $\tau' = L(\chi_V\text{St}_{l-2t+1}\nu^\frac{1}{2};\tau)$. By Theorem 4.5 of \cite{Atobe_Gan}, this implies that $\tau'$ is the first lift of $\theta_{l-2t+2}(\tau')$ on the going-up tower. This has two important consequences:
\begin{itemize}
\item $(W_n)$ is the going-up tower for the tempered representation $\tau'$. It easily follows that $(W_n)$ is also the going-up tower for the discrete series part, $\sigma'$.
\item $l(\theta_{l-2t+2}(\tau')) = l-2t$, from the conservation relation.
\end{itemize} 
Furthermore, $\tau' = L(\chi_V\text{St}_{l-2t+1}\nu^\frac{1}{2};\tau)$ also implies that $\theta_{l-2t+2}(\tau')$ is a subquotient of $\text{St}_{l-2t}\rtimes \Theta_{l-2t}(\tau)$ (cf.~Corollary \ref{cor_shaving}). Since $c_1 = 1-d_1$, $m_\phi(\chi_VS_{l-2t})$ must be odd---otherwise, by Theorem \ref{theorem_occurrence}, we would have $\theta_l(\pi) = 0$. The results of \cite{Atobe_Gan} now imply that $\chi_WS_{l-2t}$ appears with even multiplicity in any subquotient of $\Theta_{l-2t}(\tau)$. Since we have established that $l(\theta_{l-2t+2}(\tau')) = l-2t$, even multiplicity of $\chi_WS_{l-2t}$ implies (again, using the results of \cite{Atobe_Gan}) that $l(\sigma') = l-2t-2$. Using the conservation relation, we get that $\theta_{2t-l}(\sigma')$ is the first lift of $\sigma'$ on the going-up tower. It is important to note that this representation is tempered. 

We now recall the results of Muić on the possible subquotients of full theta lifts. By Theorem 6.1 of \cite{Muic_theta_discrete_Israel}, $\theta_{-l}(\sigma')$ is the unique irreducible quotient of
\[
\chi_V|\cdot|^{\frac{l-1}{2}} \times \chi_V|\cdot|^{\frac{l-3}{2}} \times \dotsm \times \chi_V|\cdot|^{\frac{l+1}{2}-t} \rtimes \theta_{2t-l}(\sigma').
\]
Any other irreducible subquotient of $\Theta_{-l}(\sigma')$ is the quotient of the standard representation 
\[
\chi_V|\cdot|^{\frac{l-1}{2}} \times \chi_V|\cdot|^{\frac{l-3}{2}} \times \dotsm \times \chi_V|\cdot|^{\frac{l+1}{2}-r} \rtimes \sigma_1,
\]
where $r < t$ and $\sigma_1$ is a tempered irreducible subquotient of $\Theta_{2r-l}(\sigma')$ (we allow $r=0$). This is proved in \cite{Muic_theta_discrete_Israel} for the symplectic-orthogonal dual pair, but the same proof, mutatis mutandis, works in the metaplectic-orthogonal case. Let us assume that some irreducible subquotient of $\Theta_{-l}(\sigma')$ other than $\theta_{-l}(\sigma')$ participates in \eqref{eq_stdmod_lift}. Recalling that $\zrep{\frac{l+1}{2}-r}{\frac{l-1}{2}}$ is the unique irreducible quotient of $|\cdot|^{\frac{l-1}{2}}  \times \dotsm \times |\cdot|^{\frac{l+1}{2}-r}$, this assumption implies that $\theta_{-l}(\pi')$ is a quotient of 
\begin{equation}
\label{eq_slight_return}
\begin{aligned}
\chi_V\Xi \times \chi_VA' \times \chi_V\Delta\times \chi_V\zrep{\frac{l+1}{2}-r}{\frac{l-1}{2}} \rtimes \sigma_1.
\end{aligned}
\end{equation}
If none of the discrete series representations $\delta_1,\dotsc, \delta_d$ which appear in $\Delta$ are equal $\text{St}_{l-2r}$, then $\zrep{\frac{l+1}{2}-r}{\frac{l-1}{2}}$ may switch places with all of them; we thus get
\begin{equation}
\label{eq_slight_return_1}
\begin{aligned}
\chi_V\Xi \times \chi_VA' \times \chi_V\zrep{\frac{l+1}{2}-r}{\frac{l-1}{2}} \rtimes \tau_1 \twoheadrightarrow \theta_{-l}(\pi')
\end{aligned}
\end{equation}
for some irreducible subquotient $\tau_1$ of $\chi_V\Delta \rtimes \sigma_1$. In the other case, when $\text{St}_{l-2r}$ appears among the $\delta_i$'s, we may write $\Delta = \Delta' \times \text{St}_{l-2r}$. We either get \eqref{eq_slight_return_1}, or
\begin{equation}
\label{eq_slight_return_2}
\chi_V\Xi \times \chi_VA' \times \chi_V\zrep{\frac{l+3}{2}-r}{\frac{l-1}{2}} \times \chi_V\text{St}_{l-2r+1}\nu^\frac{1}{2} \rtimes \tau_2 \twoheadrightarrow \theta_{-l}(\pi'),
\end{equation}
where $\tau_2$ is now an irreducible subquotient of $\chi_V\Delta' \rtimes \sigma_1$. This follows from the fact that $\text{St}_{l-2r} \times \zrep{\frac{l+3}{2}-r}{\frac{l-1}{2}}$ has only two irreducible subquotients (see Lemma \ref{lem_zel1}).

We are now ready to prove the lemma. The above discussion proves the following: if a subquotient of $\Theta_{-l}(\sigma')$ different from $\theta_{-l}(\sigma')$ participates in \eqref{eq_stdmod_lift}, this implies that $\theta_{-l}(\pi')$ is a quotient of \eqref{eq_slight_return}, which in turn leads to \eqref{eq_slight_return_1} or \eqref{eq_slight_return_2}. We prove that this is impossible using an inductive argument.

In both cases, Lemma \ref{lem_len_k} shows that (in the notation of that lemma) $\text{len}_{\frac{l-1}{2}}(\theta_{-l}(\pi')) \leq \text{len}_{\frac{l-1}{2}}(\pi)$. This enables the inductive argument. Recall that we are working under the inductive hypothesis that Theorem \ref{theorem_appearance} holds whenever $ \text{len}_{\frac{l-1}{2}}(\pi) < t$. Thus, if $\text{len}_{\frac{l-1}{2}}(\theta_{-l}(\pi'))  = t' < t$, we may use Theorem \ref{theorem_appearance} to compute $\theta_l(\theta_{-l}(\pi')) = \pi'$. By Theorem \ref{theorem_appearance}, the tempered part in the standard module of $\theta_l(\theta_{-l}(\pi'))$ would have to be equal to 
\[
\theta_{l-2t'}(\tau_i) \quad\text{or} \quad \theta_{l-2t'+2}(L(\chi_V\text{St}_{l-2t'+1}\nu^\frac{1}{2};\tau_i)), \quad \text{with} \quad i =
\begin{cases}
1, &\text{if } \eqref{eq_slight_return_1} \text{ holds}\\
2 &\text{if } \eqref{eq_slight_return_2} \text{ holds}.
\end{cases}
\]
We claim that none of these tempered representations are equal to $\theta_{l-2t+2}(\tau')$ which appears in \eqref{eq_stdmod2}. Indeed, by \cite{Atobe_Gan}, we know the lifts of these representations to level $-(l-2t+2)$; they are
\[
L(\chi_V|\cdot|^{\frac{l+1}{2}-t},\dotsc,\chi_V|\cdot|^{\frac{l+1}{2}-t'}; \tau_i) \text{ and } L(\chi_V|\cdot|^{\frac{l+1}{2}-t},\dotsc,\chi_V|\cdot|^{\frac{l+3}{2}-t'}, \chi_V\text{St}_{l-2t'+1}\nu^\frac{1}{2}; \tau_i).
\] 
Since neither of these representations is equal to $\tau'$, our claim holds. We have thus shown that $t' < t$ leads to a contradiction with our inductive hypothesis.

If $t' = t$, we use a different inductive argument. Recall that we are still working under the assumption that $\text{len}_{\frac{l-1}{2}}(\pi) = t$ and that the lowest rung of the ladder is $[c_1,d_1] = [1-d_1,d_1] = [t-\frac{l-1}{2},\frac{l+1}{2}-t]$. Moreover, by the inductive hypothesis, Theorem \ref{theorem_appearance} holds for all representations for which the relevant ladder is shorter than $t$. We now use a nested inductive argument.
Let $m$ be the number of segments of the form $[1-d_1,d_1]$ which appear in the $\alpha$-block of $\pi$. We induce on $m$ to prove that neither \eqref{eq_slight_return_1} nor \eqref{eq_slight_return_2} are possible.

The base case is $m=1$: in this situation, we actually have no segments of the form $[c_1,d_1] = [1-d_1,d_1]$ in \eqref{eq_slight_return_1} and \eqref{eq_slight_return_2} (the one from $\pi$ was shortened and ended up in $\Delta$). Therefore, by Lemma \ref{lem_len_k}, in both \eqref{eq_slight_return_1} and \eqref{eq_slight_return_2} we in fact have $\text{len}_{\frac{l-1}{2}}(\theta_{-l}(\pi')) < t$. Indeed, the proof of Lemma \ref{lem_len_k} shows that the ladder which appears in $\theta_{-l}(\pi')$ is wider than the one in $\pi$; therefore, it cannot be extended to length $t$ by a segment of the form $[c',d_1]$ with $c' > c_1$. As there are no segments of the form $[c_1,d_1]$, we deduce that $\text{len}_{\frac{l-1}{2}}(\theta_{-l}(\pi')) < t$, and this again leads to a contradiction as before.

Thus, Lemma \ref{lemma_small_theta} holds if $m=1$, and we may use the remaining arguments from Section \ref{sec_lifts} to finish the proof of Theorem \ref{theorem_appearance} (1). In other words, Theorem \ref{theorem_appearance} (1) holds when $\text{len}_{\frac{l-1}{2}}(\pi)=t$ and $m=1$ (modulo the general inductive hypothesis, which states that it holds in case when the ladder is shorter than $t$).

This allows the following (nested) induction hypothesis: Theorem 6.1 (1) is also valid whenever the length of the ladder is $t$, the last segment in the ladder is $[c_1,d_1] = [1-d_1,d_1]$, and there are strictly less than $m$ factors of the form $\drep{1-d_1}{d_1}$ in the $\alpha$-block. To perform the induction step, suppose there are $m>1$ segments of the form $[c_1,d_1]$ in the $\alpha$-block of $\pi$. It would follow from \eqref{eq_slight_return_1} or \eqref{eq_slight_return_2} that there are $m-1$ of them in the $\alpha$-block of $\theta_{-l}(\pi')$. Thus, the ladder is $\theta_{-l}(\pi')$ is of length $t$ and ends in $[c_1,d_1]$; there are $m-1$ segments of the form $[c_1,d_1]$ so we may use the inductive hypothesis: using Theorem \ref{theorem_appearance}, we compute $\theta_l(\theta_{-l}(\pi')$ and see that the resulting $\alpha$-block has exactly $m-2$ segments of the form $[c_1,d_1]$.
This contradicts the definition of $\pi'$, i.e.~\eqref{eq_stdmod2}, which has exactly $m-1$ such segments in the $\alpha$-block. Therefore, Lemma \ref{lemma_small_theta} also holds when there are $m$ segments of the form $[1-d_1,d_1]$. Again, using the arguments from Section \ref{sec_lifts}, we now prove that Theorem \ref{theorem_appearance} (1) holds in this case, thus completing the (nested) induction step.

This completes the proof of Lemma \ref{lemma_small_theta}. 
\end{proof}

\begin{proof}[Proof of Lemma \ref{lemma_unique_subquotient}]
First, notice that $\Theta_{2t-l-2}(\sigma)$ certainly contains $\tau'$ as a (sub)quotient. Therefore, it suffices to prove that the subquotient of the form $L(\chi_V\text{St}_{l-2t+1}\nu^\frac{1}{2};\tau_2')$ is unique.

We know that the multiplicity of $\chi_WS_{l-2t}$ in the parameter of $\sigma$ is even, say $2h >0$. Thus, there is an irreducible tempered representation $\sigma_1$ whose parameter does not contain $\chi_WS_{l-2t}$, such that
\[
(\chi_W\text{St}_{l-2t},h) \rtimes \sigma_1 \twoheadrightarrow \sigma,
\]
where $(\chi_W\text{St}_{l-2t},h) = \chi_W\text{St}_{l-2t} \times \dotsm \chi_W\text{St}_{l-2t}$ ($h$ times). Applying Corollary \ref{cor_theta_epi}, we get
\[
(\chi_V\text{St}_{l-2t},h) \rtimes \Theta_{2t-l-2}(\sigma_1) \twoheadrightarrow \Theta_{2t-l-2}(\sigma).
\]Note that $\Theta_{2t-l}(\sigma_1)$ is the first non-zero lift of $\sigma_1$ on this tower; it is irreducible and tempered. We now use the same result we have already used for discrete series representations: the only non-tempered irreducible subquotient of $\Theta_{2t-l-2}(\sigma_1)$ is $L(\chi_V|\cdot|^{{\frac{l+1}{2}-t}};\theta_{2t-l}(\sigma'))$. This follows from Theorem 4.1 \cite{Muic_theta_discrete_Israel}. Although the theorem is originally stated for discrete series representations, the fact that the parameter of $\sigma_1$ does not contain $\chi_WS_{l-2t}$ allows us to modify the proof so that it also applies to $\sigma_1$; we leave the simple verification of this fact to the reader.

Now let $\xi$ be an irreducible subquotient of $\Theta_{2t-l-2}(\sigma)$ of the form $L(\chi_V\text{St}_{l-2t+1}\nu^\frac{1}{2};\tau_2')$. Since $\xi$ is non-tempered, the above discussion shows that it must be a subquotient of
$
(\chi_V\text{St}_{l-2t},h) \rtimes \chi_V|\cdot|^{{\frac{l+1}{2}-t}} \rtimes \theta_{2t-l}(\sigma_1)
$. As mentioned before, $\theta_{2t-l}(\sigma_1)$ is tempered. By Theorem 4.5 of \cite{Atobe_Gan}, its parameter contains $\chi_VS_{l-2t}$. Therefore, $(\chi_V\text{St}_{l-2t},h) \rtimes \theta_{2t-l}(\sigma_1)$ is irreducible and tempered. To simplify notation, we denote this representation by $\tau''$, and we let $A = \chi_V|\cdot|^{{\frac{l+1}{2}-t}} \rtimes \tau''$. Thus, $\xi$ is a subquotient of $A$ and it remains to show that $A$ possesses only one irreducible subquotient with standard module of the form $\chi_V\text{St}_{l-2t+1}\nu^\frac{1}{2} \rtimes \tau_2'$. Using Frobenius reciprocity, we see that 
\[
\xi \hookrightarrow \chi_V\text{St}_{l-2t+1}\nu^{-\frac{1}{2}} \rtimes \tau_2'
\]
implies that $R_{P_{l-2t+1}}(\xi)$ has an irreducible quotient of the form $\chi_V\text{St}_{l-2t+1}\nu^{-\frac{1}{2}} \otimes \tau_2'$. Formulas \eqref{eq_tadic_classical} and \eqref{eq_tadic_classical2} along with Casselman's criterion shows that any subquotient of $R_{P_{l-2t+1}}(A)$ with $\GL_{P_{l-2t+1}}(F)$ acting by $\chi_V\text{St}_{l-2t+1}\nu^{-\frac{1}{2}}$ must be equal to $\chi_V\text{St}_{l-2t+1} \nu^{-\frac{1}{2}}\allowbreak \otimes\allowbreak (\chi_V\text{St}_{l-2t},\allowbreak h-1) \rtimes \theta_{2t-l}(\sigma_1)$. Therefore $\xi \cong L(\chi_V\text{St}_{l-2t+1}\nu^{\frac{1}{2}}; (\chi_V\text{St}_{l-2t},h-1) \rtimes \theta_{2t-l}(\sigma_1))$. We now show that $\xi$ appears with multiplicity one in $A$.

We let $q$ denote the Langlands quotient of $|\cdot|^{{\frac{l+1}{2}-t}} \times \text{St}_{l-2t}$, so that there is an exact sequence $0 \to \text{St}_{l-2t+1}\nu^{\frac{1}{2}} \to |\cdot|^{{\frac{l+1}{2}-t}} \times \text{St}_{l-2t} \to q \to 0$. Inducing, we get $0 \to A_1 \to A \to A_2 \to 0$, where
\begin{align*}
A_1 &= \chi_V\text{St}_{l-2t+1}\nu^{\frac{1}{2}} \times (\chi_V\text{St}_{l-2t},h-1) \rtimes \theta_{2t-l}(\sigma_1)\\
A_2 &= \chi_Vq \times (\chi_V\text{St}_{l-2t},h-1) \rtimes \theta_{2t-l}(\sigma_1)
\end{align*}
Now $\xi$ is exactly the Langlands quotient of $A_1$, so it appears in $A_1$ with multiplicity one. It thus suffices to show that $\xi$ cannot appear in $A_2$. To do this, we use another Jacquet module computation. From $\text{St}_{l-2t+1}\nu^{-\frac{1}{2}} \hookrightarrow \text{St}_{l-2t} \times |\cdot|^{{t-\frac{l+1}{2}}}$ we get
\[
\xi \hookrightarrow (\chi_V\text{St}_{l-2t},h-1) \times \chi_V\text{St}_{l-2t+1}\nu^{-\frac{1}{2}} \rtimes \theta_{2t-l}(\sigma_1) \hookrightarrow (\chi_V\text{St}_{l-2t},h) \times \chi_V|\cdot|^{{t-\frac{l+1}{2}}} \rtimes \theta_{2t-l}(\sigma_1).
\]
Frobenius reciprocity now shows
\[
\Hom(R_{P_{h(l-2t)}}(\xi), (\chi_V\text{St}_{l-2t},h) \otimes \chi_V|\cdot|^{{t-\frac{l+1}{2}}} \rtimes \theta_{2t-l}(\sigma_1)) \neq 0,
\]
so $R_{P_{h(l-2t)}}(\xi)$ has a quotient of the form $(\chi_V\text{St}_{l-2t},h) \otimes \xi'$ for some irreducible representation $\xi'$. Using Tadić's formula again to compute $\mu^*(A_2)$ shows that $R_{P_{h(l-2t)}}(A_2)$ contains no such subquotient (here we use the fact that $\theta_{2t-l}(\sigma_1)$ contains $\chi_VS_{l-2t}$ with multiplicity one). This shows that $A_2$ does not have a subquotient isomorphic to $\xi$, completing the proof.
\end{proof}

\bibliographystyle{siam}
\bibliography{bibliography}

\bigskip
\end{document}